\documentclass[11pt,reqno]{amsart}

\usepackage[foot]{amsaddr}
\usepackage{amsmath,amssymb,amsthm}
\usepackage{booktabs}
\usepackage{bm}

\usepackage{tikz}
\usetikzlibrary{arrows,shapes}
\tikzset{>=stealth}
\usepackage{float}
\floatstyle{ruled}
\newfloat{algorithm}{htb}{alg}
\floatname{algorithm}{Algorithm}
\usepackage{hyperref}

\usepackage[maxbibnames=99,
style=numeric,
doi=false,
backend=bibtex,
sorting=nyt,
sortcites,
natbib=true,
url = false,
isbn=false]{biblatex}

\theoremstyle{plain}
\newtheorem{lemma}{Lemma}
\newtheorem{theorem}{Theorem}
\newtheorem{proposition}{Proposition}
\newtheorem{corollary}{Corollary}

\theoremstyle{definition}
\newtheorem{definition}{Definition}
\newtheorem{example}{Example}
\newtheorem{problem}{Problem}
\newtheorem{assumption}{Assumption}

\theoremstyle{remark}
\newtheorem{remark}{Remark}

\let\leq\leqslant
\let\geq\geqslant

\let\mc\mathcal

\DeclareMathOperator{\wt}{weight}

\title{Maintenance scheduling in a railway corridor}

\author[Saman Eskandarzadeh]{Saman Eskandarzadeh$^{1,2}$}
\address{$^1$School of Mathematical \& Physical Sciences, University of Newcastle, Callaghan, NSW
  2308, Australia}
\address{$^2$ Institute of Transport and Logistics Studies, University of Sydney, NSW 2006, Australia}
\author[Thomas Kalinowski]{Thomas Kalinowski$^{1,3}$}
\address{$^3$School of Science and Technology, University of New England, Armidale, NSW 2351, Australia}
\author[Hamish Waterer]{Hamish Waterer$^1$}

\email{saman.eskandarzadeh@sydney.edu.au}
\email{tkalinow@une.edu.au}
\email{hamish.waterer@newcastle.edu.au}

\thanks{This research is supported by the Australian Research Council and Aurizon
Network Pty Ltd under the grant LP140101000.}

\date{\today}


\begin{document}

\begin{abstract}
  We investigate a novel scheduling problem which is motivated by an application in the Australian
  railway industry. Given a set of maintenance jobs and a set of train paths over a railway corridor with
  bidirectional traffic, we seek a schedule of jobs such that a minimum number of train paths are
  cancelled due to conflict with the job schedule. We show that the problem is NP-complete in
  general. In a special case of the problem when every job under any schedule just affects one train
  path, and the speed of trains is bounded from above and below, we show that the problem can be
  solved in polynomial time. Moreover, in another special case of the problem where the traffic is
  unidirectional, we show that the problem can be solved in time $O(n^4)$.
\end{abstract}

\maketitle

\section{Introduction}\label{sec:intro}
Australia has a large operational heavy railway network which comprises approximately 33,355
route-kilometres. This network accounted for approximately 55 percent of all freight transport
activity in Australia in the financial year 2013-14, almost 367 billion tonne-kilometres which was
up 50 percent from 2011-12 (\cite{BITRE2016}). To prevent long unplanned interruptions in the
service to customers, a proper maintenance and renewal program for the network infrastructures is
required. The objective is to schedule planned maintenance and asset renewal jobs in such a way that
their impact on the capacity that will be provided to customers is minimised while at the same time
keeping the infrastructure in good working condition. An effective planned maintenance and renewal
schedule reduces the frequency with which disruptive reactive maintenance is needed.

We investigate a planned maintenance and asset renewal scheduling problem on a railway corridor with
train traffic in both directions. Potential train journeys are represented by train paths, where a
train path is specified by a sequence of (location,time)-pairs, and we distinguish between up- and
down-paths, depending on the direction of travel. Necessary maintenance and renewal activities, or
work, are specified by a release time, a deadline, a processing time and a location. Scheduling work
at a particular time has the consequence that the train paths passing through the corresponding
location while the work is carried out have to be cancelled. An instance of the problem is given by
a set of train paths and a set of work activities, and the task is to schedule all the work such
that the total number of cancelled paths is minimised.

There is a vast literature on scheduling problems and transportation networks. However, the
interactions of scheduling problems and transportation networks in contexts such as the railway
industry has not been studied thoroughly. \citet{BolandKWZ2014} study the problem of scheduling
maintenance jobs in a network. Each maintenance job causes a loss in the capacity of the network
while it is being performed. The objective is to minimise this loss, or equivalently, maximize the
capacity over time horizon, while ensuring that all jobs are scheduled. They model the problem as a
network flow problem over time. This problem and its variants are investigated in
\cite{BolandKWZ2013,BolandKKK2014,BolandKK2015,BolandKK2016,AbedCDGMMRR2017}. Our work is different
from these previous works in that we model the capacity by discrete train paths whereas in the
network flow models capacity is approximated by continuous flows over time.

The second stream of related research studies the problem of scheduling jobs on a single machine or
multiple parallel machines with the goal of minimising the total busy time of machines (see, for
example, \cite{ChangGK2014,KhandekarSST2015,KoehlerK2017,ChangKM2017}). This problem is closely
related to a special case of our problem in which there is unidirectional traffic.

The third related stream is the body of research which explores variants of the hitting set problem
(see, for example, \cite{HassinM1991,MustafaR2010,FeketeHMPP2018}). Some of the results in this
paper are due to the close connections between the maintenance scheduling problem, machine
scheduling and the hitting set problem. As we show in Section~\ref{sec:MIP}, the maintenance
scheduling problem can be formulated as a set covering problem.

In Section~\ref{sec:problem}, we formally introduce the maintenance scheduling problem, and we prove
that it is NP-complete by a reduction from a variant of the hitting set
problem. Section~\ref{sec:exact_algorithm} contains a dynamic programming algorithm for solving the
maintenance scheduling problem. In Section~\ref{sec:MIP}, we present two integer programming (IP) formulations and
compare their LP-relaxations under some additional assumptions. In Section~\ref{sec:unidirectional},
we investigate a special case of the problem with unidirectional traffic and prove that the problem
can be solved in polynomial time by dynamic programming. Making an additional
assumption on the set of jobs, we improve the runtime bound from $O(n^4)$ to $O(n^3)$, where $n$ is the number of
jobs. In Section~\ref{sec:bounded_intersections}, we investigate another special case of the problem with
two main assumptions: 1) the minimum speed of trains, and 2) the length of the corridor is
bounded. We show that this special case can be solved in quadratic time by formulating it as a
shortest path problem. In Section~\ref{sec:computation}, we compare the computational performance of
the IP formulations.

\section{Problem description}\label{sec:problem}
We are given a set $\mc{J}$ of $n$ maintenance jobs, and a set $\mc{P}$ of train paths. Each job
$j \in \mc{J}$ is specified by its earliest start time $r_j$, its latest finish time $d_j$, its
processing time $p_j$, its start location $l^s_j$ and its end location $l^e_j$, all of which are
non-negative rational numbers. A train path represents the movement of a train through the railway
corridor whose length we denote by $L$, so that the corridor can be represented by the interval
$[0,L]$. In this work we make the assumption that trains move from end to end with constant velocity
and that the train paths in either direction are regularly distributed over time. More precisely,
the set $\mc P$ of train paths comes with a partition into the set
$ \mc{P}^{u}=\{\ell^u_1, \dots, \ell^u_{m}\} $ of \emph{up-paths} and the set
$ \mc{P}^{d}=\{\ell^d_1, \dots, \ell^d_{m}\} $ of \emph{down-paths}, where path $\ell_i^u$ starts at
location $0$ at time $i\Delta$ and arrives at location $L$ at time $i\Delta+\delta$, while path
$\ell_i^d$ starts at location $L$ at time $i\Delta$ and arrives at location $0$ at time
$i\Delta+\delta$.  Here $\Delta$ and $\delta$ are positive rational numbers representing the headway and
the total travel time, respectively. Every job $j$ occupies or possesses the corridor at a location
which starts at $ \ell^s_j $ and ends at $ \ell^e_j $ for duration $p_j$. If a location is
possessed by a job in a certain time interval, then any path that passes through this location during that
time interval has to be cancelled. Note that different jobs can possess the same location at the
same time. The decision in the maintenance scheduling problem is to find start times of jobs in such
a way that the number of cancelled paths is minimised. A solution is given by a vector
$s=(s_j)_{j \in \mc{J}}$ of start times.

An instance of the problem can be represented geometrically in the plane, where the horizontal and
vertical axes represent time and location, respectively. Let $[(x_1,y_1),(x_2,y_2)]$ denote a line
segment connecting two points $(x_1,y_1)\in \mathbb{R}^2 $ and $(x_2,y_2) \in \mathbb{R}^2$. Train
paths can be identified with line segments: $\ell^u_i=[(i\Delta,0),(i\Delta+\delta,L)]$ and
$\ell^d_i=[(i\Delta,L),(i\Delta+\delta,0)]$. Jobs correspond to rectangular boxes as illustrated in
Figure~\ref{Figure:instance1}.
If job $j \in \mc{J}$ starts at time $s_j \in [r_j,d_j-p_j]$ then there is a set
$R \subseteq \mc{P}$ of paths that have to be cancelled as they would pass through the location
during the interval $[s_j, s_j+p_j]$. We call the path set $R$ a \emph{possession}. Let $\mc{R}_j$
denote the collection of all such possessions for job $j$. The problem of finding an optimal start
time vector $s$ is equivalent to selecting a possession $R_j\in\mc{R}_j$ for each job $j$ such that
the cardinality of their union is minimised.
\begin{figure}[htbp]
\centering   	
\begin{tikzpicture}[yscale=0.3,xscale=.4]
	\draw[->, very thick]  (-1,0) -- (32,0) node[anchor=north] {Time};
	\draw[->, very thick]  (0,-1) -- (0,22) node[anchor=east] {Location};
	\foreach \x in {1,...,10}
		{
	  		\draw (\x*2,0) -- +(10,20);%
	  		\draw (\x*2,20) -- +(10,-20); %
	  			\draw (\x*2+10,20) node[anchor=south] {$\ell_{\x}^u$}; %
	  		\draw (\x*2+9,-0.4) node[anchor=south] {$\ell_{\x}^d$};
	  	}
	\draw[dotted]  (4,4) rectangle +(10,1);
	\filldraw[fill=orange!90]  (8.5,4) +(-1,0) rectangle +(4,1);
	\draw[dotted]  (10,8) rectangle +(15,2);
	\filldraw[fill=orange!90]  (17,8) +(-1,0) rectangle +(2,2);
	\draw[dotted]  (8,12) rectangle +(8,3);
	\filldraw[fill=orange!90]  (11.5,12) +(-1,0) rectangle +(0.1,3);
	\draw[dotted]  (11,14) rectangle +(13,2);
	\filldraw[fill=orange!90]  (17.2,14) +(-1,0) rectangle +(1,2);	
	\draw (0,20) node[anchor=east] {$L$};
	\draw[dashed] (14,6) -- (14,0) node[anchor=north] {$d_j$};
	\draw[dashed] (4,6) -- (4,0) node[anchor=north] {$r_j$};
	\draw[<->] (7.5,6) -- +(2.5,0) node[anchor=south] {$p_j$} --(12.5,6);
	\draw[dashed] (4,4) -- (0,4) node[anchor=east] {$l^s_j$};
	\draw[dashed] (4,5) -- (-1.5,5) node[anchor=east] {$l^e_j$}; 
\end{tikzpicture}
\caption{An instance of the maintenance scheduling problem with four jobs. The dotted boxes
  represent the jobs, and the filled boxes indicate a feasible solution. The set of cancelled paths
  is
  $\{\ell^d_1,\,\ell^d_2,\,\ell^d_4,\,\ell^d_6,\,\ell^d_7,\,\ell^d_8,\,\ell^u_2,\,\ell^u_3,\,\ell^u_4,\,\ell^u_5,\,\ell^u_6,\,\ell^u_7\}$.}\label{Figure:instance1}
\end{figure}
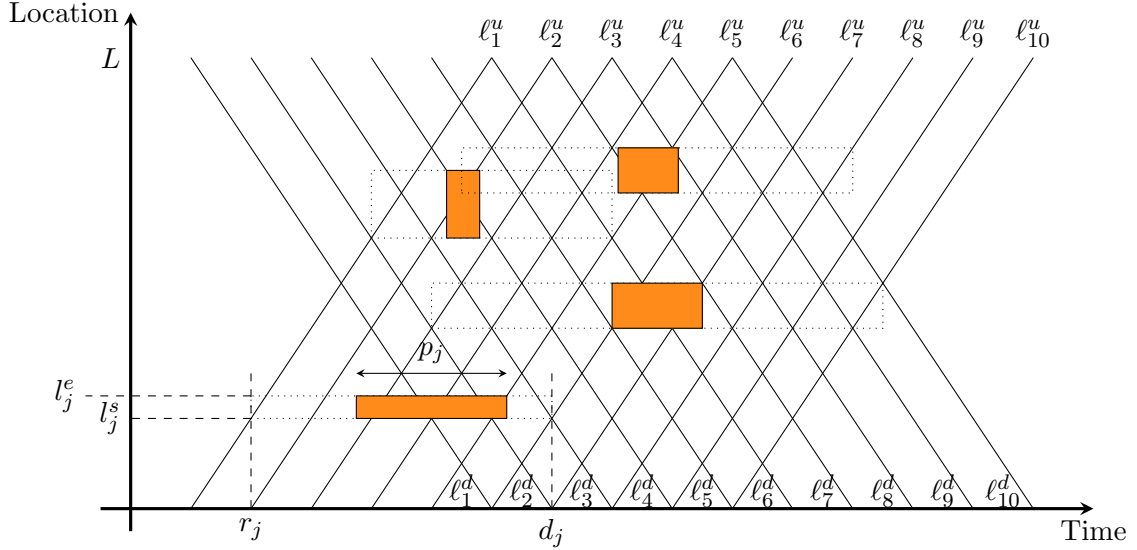

If $\emptyset\in\mc{R}_j$ then the job $j$ can be
removed from $\mc{J}$ without changing the problem. If $R\subset R'$ for two elements
$R,R'\in\mc{R}_j$ then $R'$ can be removed from $\mc{R}_j$, because in any feasible solution $R'$
can be replaced by $R$ without increasing the objective value. If there are two jobs $j,j'$ such
that for every $R\in\mc{R}_j$ there exists an $R'\in\mc{R}_{j'}$ with $R'\subseteq R$ then $j'$ can
be removed from $\mc{J}$ without changing the optimal objective value. In summary, we can make the
following assumptions:
\begin{enumerate}
\item For all $j\in\mc{J}$, $\emptyset\not\in\mc{R}_j$.
\item For all $j\in\mc{J}$ and distinct $R,R'\in\mc{R}_j$, $R\not\subset R'$.
\item For all $j,j'\in\mc{J}$, there exists $R\in\mc{R}_j$ such that $R'\nsubseteq R$ for all $R'\in\mc{R}_{j'}$.
\end{enumerate}

We conclude this section by establishing that the maintenance scheduling problem is NP-complete.

\begin{theorem}\label{thm:hardness}
  The maintenance scheduling problem is NP-complete.
\end{theorem}
\begin{proof}
  Since a given solution can be evaluated in polynomial time, the problem is in the class $NP$.  We
  prove NP-hardness by a reduction from the \textit{minimum hitting horizontal unit segments by
    axis-parallel lines} problem (HHP)~(\cite{HassinM1991}), which can be specified as follows:
  \begin{description}
  \item[Instance.] An instance is given by two sets $A = \{a_1, \dotsc, a_n\}$ and
    $B = \{b_1, \dotsc, b_n\}$ of integers.
  \item[Solution.] A feasible solution is given by two sets $X$ and $Y$ of positive integers such
    that for every $j \in [n]$, $X \cap \{a_j, a_j + 1\} \neq \emptyset$ or
    $Y \cap \{ b_j\} \neq \emptyset$. In other words, for every $ j \in [n]$, $X$ has to contain
    $a_j$ or $a_j +1$ or $Y$ has to contain $b_j$.
  \item[Objective.] Minimize $\lvert X\rvert+\lvert Y\rvert$.
  \end{description}
  Let $(A,B)$ be an instance of the HHP. In order to construct the corresponding instance of the
  \textit{maintenance scheduling} problem (MSP), set
  \begin{align*}
    L &= 2 + \max\{\lvert a_j - b_j\rvert\,:\,j \in [n]\}, & m &= 1 + \max\{\max_{j \in  [n] } a_j,\,\max_{j \in  [n]} b_j\}, 
  \end{align*}
  and $\Delta =1$, $\delta = L$. This defines the path sets
  $\mc{P}^u$ and $\mc{P}^d$. In order to define the jobs, we note that the intersection of the line
  segments $[(a_j, 0), (a_j + L,L)]$ and $[(b_j,L),(b_j + L,0)]$ is
  \[((L + a_j + b_j)/2,\,(L - a_j +b_j)/2).\]
  Let $\mc{J}$ be the set of jobs with the following parameters:
  \begin{align*}
    r_j &= \frac{L + a_j + b_j}{2} - \frac{1}{2},& d_j &= \frac{L + a_j + b_j}{2} - 1,\\
    l_j^s=l_j^e  &= \frac{L - a_j + b_j}{2} - \frac{1}{4},  &p_j &= \frac{3}{5}.
  \end{align*}
  These parameters are chosen in such a way that
  $\mc{R}_j = \{\{\ell^u_{a_j}\},\{\ell^d_{b_j}\},\{\ell^u_{a_j+1}\}\}$ for every $j \in
  \mc{J}$. Let the above instance be denoted by $(\mc{J}, L, m, \delta, \Delta)$. We now argue that
  the HHP instance $(A,B)$ has a solution with $\lvert X\rvert+\lvert Y\rvert\leq K$ if and only if
  the MSP instance $(\mc{J}, L, m, \delta, \Delta)$ has a solution with objective value less than or
  equal to $K$.
  \begin{description}
  \item[\textbf{HHP} $\Longrightarrow$ \textbf{MSP}] Let $(X,Y)$ be a feasible solution for the HHP
    instance $(A,B,k)$ with $\lvert X\rvert+ \lvert Y\rvert \leq k$. We define a feasible collection
    of possessions $(R_j)_{j \in \mc{J}}$ for the MSP instance $(\mc{J}, L, m, \delta)$ as follows:
    \[ R_j = \begin{cases}
        \{\ell_{a_j}^u\}  & \text{if } a_j \in X, \\
        \{\ell_{b_j}^d\}  & \text{if } a_j \notin X,\,b_j \in Y,  \\
        \{\ell_{a_j+1}^u\} & \text{if } a_j \notin X,\,b_j \notin Y,\,a_j+1 \in X.
      \end{cases}
    \]
    In particular, $R_j = \{\ell_a^u\}\implies a \in X$, and
    $R_j = \{\ell_{b}^d\} \implies b \in Y$. Together with
    $\bigcup_{j \in \mc{J}}R_j \subseteq \mc{P}^u \cup \mc{P}^d$,
    $\mc{P}^u \cap \mc{P}^d = \emptyset$, and $|X| + |Y| \leq K$, this implies
    \[ \left| \bigcup_{j \in \mc{J}}R_j \right| = \left| \bigcup_{j \in \mc{J}}R_j \cap \mc{P}^u
      \right| +\left| \bigcup_{j \in \mc{J}}R_j \cap \mc{P}^d \right| \leq |X| + |Y| \leq K. \]
  \item[\textbf{MSP} $\Longrightarrow$ \textbf{HHP}] Let $(R_j)_{j \in \mc{J}}$ be a solution to
    the MSP instance $(\mc{J}, L, m, \delta, \Delta)$ with
    $\left|\bigcup_{j \in \mc{J}}R_j \right| \leq K$. Define $(X,Y)$ for the corresponding HHP
    instance $(A,B)$ as follows:
    \begin{align*}
      X &= \{a \in [m]: \ell_a^u \in \bigcup_{j\in\mc{J}}R_j\}, & Y &= \{b \in [m]: \ell_b^d \in \bigcup_{j\in\mc{J}}R_j\}.
    \end{align*}
    The feasibility of $(X,Y)$ follows from the feasibility of the solution $(R_j)_{j \in \mc{J}}$:
    \[ \forall j \in \mc{J} \ R_j \in\left\{ \{\ell^u_{a_j}\},\,\{\ell^d_{b_j}\},\,
        \{\ell^u_{a_j+1}\}\right\} \implies \forall j \in [n] \ a_j \in X \text{ or } a_j+1 \in X
      \text{ or } b_j \in Y.\]
    We conclude
    \[ |X| + |Y| = \left| \bigcup_{j \in \mc{J}}R_j \cap \mc{P}^u \right| +\left| \bigcup_{j \in
          \mc{J}}R_j \cap \mc{P}^d \right| = \left|\bigcup_{j \in \mc{J}}R_j \right| \leq
      K.\qedhere\]
  \end{description}
\end{proof}

\section{An exact algorithm}\label{sec:exact_algorithm}
In this section, we present a dynamic programming algorithm for the maintenance scheduling
problem. Before proceeding, we need to define more notation. We define the following sets for
$P, X \subseteq \mc{P}$:
\begin{itemize}
\item $R(X,P)$ is the set of paths in $P$ that are on the right-hand side of all paths in $X$:
  \[R(X,P) = \{\ell_i^u\in P\,:\,\ell_j^u\in X\implies j<i\}\cup\{\ell_i^d\in P\,:\,\ell_j^d\in
    X\implies j<i\}.\]
\item $L(X,P)$ is the set of paths in $P$ that are on the left-hand side of all paths in $X$:
  \[L(X,P) = \{\ell_i^u\in P\,:\,\ell_j^u\in X\implies j>i\}\cup\{\ell_i^d\in P\,:\,\ell_j^d\in
    X\implies j>i\}.\]
\item $J(P)$ is the set of jobs that possess at least one path in $P$ regardless of their start
  times:
  \[J(P) =\{j \in \mc{J}\,:\, R \cap P \neq \emptyset\text{ for all }R\in\mc{R}_j\}.\]
\item $J_M(X,P)$ is the set of jobs in $J(P)$ that can be scheduled such that the cancelled
  paths in $P$ are all in $X$:
  \[J_M(X,P) =\{j \in J(P)\,:\, R \cap P \subseteq X\text{ for some }R \in \mc{R}_j\}.\]
\end{itemize}
\begin{lemma}\label{lem:remaining_jobs}
  For all $X,P\subseteq\mc{P}$, $J(P\setminus X)=J(P)\setminus J_M(X,P)$.
\end{lemma}
\begin{proof}
  Let $j\in\mc{J}$ be an arbitrary job. The lemma is a consequence of the following equivalences:
  \begin{multline*}
    j\in J(P\setminus X)\iff R\cap(P\setminus X)\neq\emptyset\text{ for all } R\in\mc{R}_j\\
    \iff R\cap P\neq\emptyset\text{ and }R\cap P\not\subseteq X\text{ for all } R\in\mc{R}_j \iff
    j\in J(P)\setminus J_M(X,P). \qedhere
  \end{multline*}
\end{proof}
Let $c(P)$ denote the minimum number of cancelled paths in the set $P$ for scheduling jobs in
$J(P)$, or formally,
\[c(P)=\min\left\{\left\lvert\bigcup_{j\in J(P)}(R_j\cap P)\right\rvert\,:\,R_j\in\mc{R}_j\text{ for
    all }j\in J(P)\right\}.\]
In particular, $c(\mc{P})$ is the optimal objective value for the complete problem.
\begin{lemma}\label{lem:recurrence}
  The optimal value $c(\mc{P})$ (and a corresponding optimal solution) can be computed by the
  following recursion:
  \[c(P)=
    \begin{cases}
      \min\{\lvert X \cap P\rvert + c(P\setminus X)\,:\,X\in\mc{R}_{j'}\}\text{ for any }j'\in J(P) & \text{if }J(P)\neq\emptyset, \\
      0 & \text{if }J(P)=\emptyset.
    \end{cases}
\]
\end{lemma}
\begin{proof}
  The case $J(P)=\emptyset$ is trivial, so we can assume $J(P)\neq\emptyset$ and pick an arbitrary
  $j'\in J(P)$. Note that for fixed $R_{j'}=X\in\mc{R}_{j'}$, the minimum in the definition of
  $c(P)$ is obtained by choosing $R_j\in\mc{R}_j$ with $R_j\cap P\subseteq X$ for all
  $j\in J_M(X,P)$.  As a consequence,
  \begin{multline*}
    c(P)=\min\left\{\left\lvert\bigcup_{j\in J(P)}(R_j\cap P)\right\rvert\,:\,R_j\in\mc{R}_j\text{
        for all }j\in J(P)\right\}\\ = \min\limits_{X\in\mc{R}_{j'}}\left(\lvert X\cap
      P\rvert+\min\left\{\left\lvert\bigcup_{j\in J(P\setminus X)}(R_{j}\cap (P\setminus
          X))\right\rvert\,:\,R_j\in\mc{R}_j\text{ for
          all }j\in J(P\setminus X)\right\}\right)\\
    =\min\limits_{X\in\mc{R}_{j'}}\left(\lvert X\cap P\rvert+c(P\setminus X)\right).\qedhere
\end{multline*}
\end{proof}
Let $J_R(X,P)$ and $J_L(X,P)$ be the sets of jobs in $J(P\setminus X)$ such that all the canceled
path in $P\setminus X$ are in $R(X,P)$ and $L(X,P)$, respectively:
\begin{align*}
  J_R(X,P) &= \{j \in J(P\setminus X)\,:\,(Y \cap P)\setminus X \subseteq R(X,P)\text{ for all
             }Y\in\mc{R}_j\},\\
  J_L(X,P) &= \{j \in J(P\setminus X)\,:\,(Y \cap P)\setminus X \subseteq L(X,P)\text{ for all
    }Y\in\mc{R}_j\}.
\end{align*}
\begin{definition}\label{def:splitting_job}
  A job $j \in J(P)$ is called \emph{splitting} if $J(P\setminus X)= J_L(X,P)\cup J_R(X,P)$ for all
  $X\in\mc{R}_j$. Let $J'(P)$ denote the set of splitting jobs in $J(P)$.
\end{definition}
\begin{lemma}\label{lem:recursion_splitting_job}
  Let $j\in J'(P)$ be a splitting job for the set $P$. Then, for every $X\in\mc{R}_j$,
  \[c(P\setminus X)=c(R(X,P))+c(L(X,P)).\]
\end{lemma}
\begin{proof}
  It follows from the definitions that for all $j'\in J_R(X,P)$, $j''\in J_L(X,P)$, $Y'\in\mc{R}_j$
  and $Y''\in\mc{R}_{j''}$,
  \[(Y'\cap(P\setminus X))\cap(Y''\cap(P\setminus X))=\emptyset.\]
  As a consequence, for every choice of $R_{j'}\in\mc{R}_{j'}$ for the jobs $j'\in J(P\setminus X)$,
  \[\left\lvert\bigcup_{j'\in J(P\setminus X)}(R_{j'}\cap (P\setminus X))\right\rvert=\left\lvert\bigcup_{j'\in
        J_R(X,P)}(R_{j'}\cap (P\setminus X))\right\rvert+\left\lvert\bigcup_{j'\in
        J_L(X,P)}(R_{j'}\cap (P\setminus X))\right\rvert,\] and therefore,
  \begin{multline*}
    c(P\setminus X)=\min\left\{\left\lvert\bigcup_{j'\in J(P\setminus X)}(R_{j'}\cap (P\setminus
        X))\right\rvert\,:\,R_{j'}\in\mc{R}_{j'}\text{ for all }j'\in J(P\setminus X)\right\}\\
    =\min\left\{\left\lvert\bigcup_{j'\in J_R(X,P)}(R_{j'}\cap (P\setminus
        X))\right\rvert\,:\,R_{j'}\in\mc{R}_{j'}\text{ for all }j'\in J_R(P\setminus X)\right\}\\
    +\min\left\{\left\lvert\bigcup_{j'\in J_L(X,P)}(R_{j'}\cap (P\setminus
        X))\right\rvert\,:\,R_{j'}\in\mc{R}_{j'}\text{ for all }j'\in J_L(P\setminus X)\right\}\\
    =c(R(X,P))+c(L(X,P)).\qedhere
  \end{multline*}
\end{proof}
The next theorem is an immediate consequence of Lemmas~\ref{lem:recurrence}
and~\ref{lem:recursion_splitting_job}.
\begin{theorem}\label{thm:recursion_1}
  The optimal value $c(\mc{P})$ (and a corresponding optimal solution) satisfies the
  following recursion:
  \[c(P)=
    \begin{cases}
      0 & \text{if }J(P)=\emptyset,\\
      \min\{\lvert X \cap P\rvert + c(P\setminus X)\,:\,X\in\mc{R}_j\}\\
      \qquad\qquad\text{for an arbitrary }j\in J(P) & \text{if }J'(P)=\emptyset\neq J(P),\\
      \min\{\lvert X \cap P\rvert+c(R(X,P))+c(L(X,P))\,:\,X \in \mc{R}_j\}\\
      \qquad\qquad\text{for an arbitrary }j\in J'(P) & \text{if }J'(P)\neq\emptyset.
    \end{cases}
\]
\end{theorem}
This recursion can be turned into an algorithm for solving the maintenance scheduling problem,
provided we can identify a splitting job. In Section~\ref{sec:unidirectional} we describe how this
can be done efficiently under the assumption that there are only up-paths (or equivalently, only
down-paths). The algorithm for the bidirectional case with some additional assumptions, which is presented in
Section~\ref{sec:bounded_intersections}, can be interpreted as being based on a refinement of the
splitting concept.

\section{Integer programming formulations}\label{sec:MIP}
In this section, we present two integer programming formulations for the maintenance scheduling
problem: a path indexed formulation and a set covering formulation.

For every path $\ell \in \mc{P}$, let $y_{\ell}$ be a binary variable which takes value one if and
only if path $\ell$ is cancelled. For every $R\in\mc{R}=\cup_{j \in \mc{J}}\mc{R}_j$, let $x_{R}$ be
a binary variable where $x_R=1$ indicates that the paths in $R$ are cancelled. We obtain the
following path indexed model (PIM):
\begin{align}
  \text{[PIM]}\qquad  \text{Minimise }  \sum_{\ell \in \mc{P}} &y_{\ell}\ \text{subject to} \nonumber \\
  \sum_{R \in \mc{R}_j}x_R &\geq 1  && \text{for all } j \in \mc{J}, \label{eq:PIM_1} \\
  y_{\ell} &\geq x_R  && \text{for all } R \in \mc{R},\ \ell \in R, \label{eq:PIM_2}  \\
  x_R &\in \{0,1\} && \text{for all } R \in \mc{R}, \\
  y_{\ell}  &\in \{0,1\} && \text{for all } \ell \in \mc{P}.
\end{align}

In order to describe the set covering formulation, we need the following notion of adjacency between
train paths.
\begin{definition}\label{def:adjacent}
  Two distinct paths $\ell$ and $\ell'$ in $\mathcal P$ are called \textit{adjacent} if there exists a
  horizontal line segment $l=[(x_1,y),(x_2,y)]$ such that $\ell$ and $\ell'$ are the only paths in
  $\mathcal P$ which $l$ intersects.  
\end{definition}
For example, in Figure~\ref{Figure:instance1}, path $\ell^u_1$ is adjacent to paths
$\ell^d_1$, $\ell^d_2$, $\ell^d_3$, $\ell^d_4$, $\ell^d_5$, and $\ell^d_6$. Let $G=(\mathcal P,E)$ be the
graph with node set $\mathcal P$ in which there is an edge between $\ell$ and $\ell'$ if and only if
they are adjacent. 
\begin{definition}\label{def:span}
  A set $S\subseteq\mathcal P$ is a \emph{span}, if it induces a connected subgraph of $G$.
\end{definition}
A feasible solution for the maintenance scheduling problem, that is a collection
$(R_j)_{j\in\mathcal J}$ of possessions, corresponds to a set of pairwise disjoint spans in $G$ such
that
\begin{enumerate}
\item every possession in the solution is a subset of a unique span, and
\item the union of spans is equal to the union of possessions in the solution.
\end{enumerate}

Now, we can present a set covering formulation for the problem. We denote the set of all spans for
an instance of the maintenance scheduling problem by $F_0$. For each job $j \in \mc{J}$, $B_j$
denotes the set of all the spans in $F_0$ which cover job $j$ or more formally:
$B_j = \{S \in F_0: \exists R \in \mc{R}_j \ R \subseteq S\}$. We introduce a binary variable $x_S$
for every $S\in F_0$, where $x_S=1$ indicates that the paths in $S$ are cancelled. The set cover
model (SCM) is given by
 \begin{align}
   \text{[SCM] } \text{Minimise} \sum_{S \in F_0}&\lvert S\rvert x_S \qquad\text{subject to}\nonumber \\
                                  \sum_{S \in B_j}x_S &\geq 1   &&\text{for all } j \in \mc{J},  \label{eq:set_cover}\\
                                  x_S &\in \{0,1\}    &&\text{for all } S \in F_0. \label{eq:binary_xS}  									
\end{align}      
The number of variables in [SCM] depends on the number of intersections of paths, the number of
paths, and the size of the possessions. If the paths are pairwise disjoint, then the number of spans
is $\lvert\mathcal P\rvert(\lvert\mathcal P\rvert+1)/2$. Another interesting case is that all
possessions have size one, in which case we can restrict the problem to the $\lvert\mathcal P\rvert$
variables corresponding to the singleton spans. The number of intersections is maximal when every
up-path intersects every down-path. Then the graph $G$ is a complete graph and the number of spans
is $2^{\lvert\mathcal P\rvert}-1$.

\section{The unidirectional case}\label{sec:unidirectional}
In this section, we consider the following special case of the maintenance scheduling problem.
\begin{definition}\label{def:unidirectional}
  The \emph{unidirectional} maintenance scheduling problem is the variant in which we have only paths
  in one direction, say only up-paths.
\end{definition}
For the unidirectional problem we can omit the upper index $u$ or $d$ on the paths, and simply write
$\mathcal P=\{\ell_1,\dots,\ell_m\}$. For $j \in \mc{J}$, let $r'_j$ and $d'_j$ be the indices of the leftmost and rightmost path
in the set $\mc{R}^u_j:=\bigcup_{R \in \mc{R}_j} R$:
\begin{align*}
 r'_j &= \min\{i : \ell_i \in \mc{R}^u_j\},& d'_j &= \max\{i : \ell_i \in \mc{R}^u_j\}.
\end{align*}
In this setting every non-dominated possession for a job $j$ contains the same number of
paths, so that we can use paths as the unit for processing time, and let $p'_j$ denote the
cardinality of the elements of $\mathcal R_j$. In other words, $\mc{R}_j$ has the form
  \[\mc{R}_j=\left\{ \left[\ell_{r'_j},\,\ell_{r'_j+p'_j-1}\right],\ \left[\ell_{r'_j+1},\,\ell_{r'_j+p'_j}\right],\,
      \dots,\ \left[\ell_{d'_j-p'_j+1},\,\ell_{d'_j}\right]\right\}.\]
If there are two jobs $j,j'\in\mathcal J$ with $p'_{j'}\leq
p'_j$ and $r'_j\leq r'_{j'}$ and $ d'_{j'}\leq d'_j$, then the job $j'$ can be removed from the problem
because it can always be scheduled in the shadow of job $j$. As a consequence we assume without loss
of generality,
for all $j,j'\in\mathcal J$
\begin{equation}\label{eq:nondom}
  p'_{j'}\leq p'_j\implies r'_j>r'_{j'}\text{ or } d'_j<d'_{j'}
\end{equation}
In particular,
\begin{equation}\label{eq:nondom_equal}
  p'_{j'}= p'_j\implies r'_j<r'_{j'}<d'_j<d'_{j'}\text{ or } r'_{j'}<r'_{j}<d'_{j'}<d'_{j}.
\end{equation}

\subsection{A polynomial time algorithm}\label{subsec:unidir_interval_poly_time_algorithm}
The unidirectional maintenance scheduling problem can be interpreted as a special case of
\emph{real-time scheduling to minimise machine busy times}~\cite{KhandekarSST2015}. In this problem,
jobs that are given by release time, due date, processing time and demand for machine capacity have
to be scheduled on machines which can process multiple jobs at the same time subject to a machine
capacity, and the objective is to minimise the total busy time of all machines, where a machine is
busy whenever it processes at least one job. The unidirectional maintenance scheduling problem
corresponds to the case where the machine capacity is infinite, which was shown to be solvable in
polynomial time in~\cite[Theorem 3.2]{KhandekarSST2015}. More precisely, for the unidirectional case
the dynamic programming algorithm described in Section~\ref{sec:exact_algorithm} turns out to be
essentially the algorithm for the real-time scheduling problem with infinite machine capacity
described in~\cite{KhandekarSST2015}. The definitions of the sets $R(X,P)$ and $L(X,P)$ from
Section~\ref{sec:exact_algorithm} simplify as follows. For a set $X\subseteq\mathcal P$, let
$i_0(X)=\min\{i\,:\,\ell_i\in X\}$ and $i_1(X)=\max\{i\,:\,\ell_i\in X\}$
\begin{align*}
R(X,P) &= \{\ell_i\in P\,:\,i> i_1(X)\}, & L(X,P) &= \{\ell_i\in P\,:\,i< i_0(X)\}.
\end{align*}

A crucial observation is that jobs of maximal length are splitting.
\begin{lemma}\label{lem:splitting_job}
  Let $P\subseteq\mc{P}$ be a set of paths, and let $j\in J(P)$ be a job of maximal length, (i.e.,
  $p'_j\geq p'_{j'}$ for all $j'\in J(P)$). Then $j$ is a splitting job for $P$.
\end{lemma}
\begin{proof}
  Suppose $j$ is not a splitting job. Then there exist $X=[\ell_a,\ell_b]\in\mc{R}_j$ and
  \[j'\in J(P\setminus X)\setminus(J_L(X,P)\cup J_R(X,P)).\]
  As a consequence, there exist $Y_1,Y_2\in\mc{R}_{j'}$ with $Y_1\cap R(X,P)\neq\emptyset$ and
  $Y_2\cap L(X,P)\neq\emptyset$. If $Y_1=Y_2$, then $Y_1\supsetneq X$ which contradicts the
  maximality of $j$. If $Y_1\neq Y_2$, then $[\ell_a,\ell_{a+p'_{j'}-1}]\in\mc{R}_{j'}$, and the
  maximality of $j$ implies $[\ell_a,\ell_{a+p'_{j'}-1}]\subseteq X$. But then $j'\in J_M(X,P)$
  which contradicts $j'\in J(P\setminus X)$.
\end{proof}
The next lemma restricts the set of possible start paths for spans and possessions. For this purpose, let
\begin{align*}
  A &= \{d'_j-p'_j+1\,:\,j\in\mathcal J\}, \\
  B_i &= \{i+p'_{j}-1\,:\,j\in\mathcal J\text{ with }r'_j\leq i\leq d'_j-p'_j+1\} &&\text{for }i\in A, \\
  C &= \{r'_{j}+p'_{j}-1\,:\,j\in\mathcal J\}.
\end{align*}
\begin{lemma}[Lemma 3.3 in \cite{KhandekarSST2015}]\label{lem:span_properties}
  There exists an optimal collection of disjoint spans in such a way that every span
  $[\ell_i,\ell_k]$ satisfies $i\in A$ and $k\in B_i\cup C$.
  Moreover, there exists an optimal solution in which for each job $j \in \mc{J}$, its possession
  starts at path $\ell_{r'_j}$ or path $\ell_{i}$ for some $i\in A$.
\end{lemma}

It follows from Lemma~\ref{lem:span_properties} that we can restrict our attention to solutions in
which the possession for job $j\in\mathcal J$ is chosen from the set
\[\mc{R}^e_j=\{[\ell_a,\ell_b] \in
  \mc{R}_j\,:\,a\in\{r'_j\}\cup\{d'_{j'}-p'_{j'}+1\,:\,j'\in\mathcal J\}\}.\]
\begin{lemma}[Lemma 3.4 in \cite{KhandekarSST2015}]\label{lem:unidir_recursion}
  The optimal value $c(\mc{P})$ can be computed using the following recursion. If $J(P)=\emptyset$, then $c(P)=0$. Otherwise,
\begin{equation}\label{eq:unidir_recursion}
  c(P)=\min_{X \in \mc{R}^e_j}\{\lvert X \cap P\rvert+c(R(X,P))+c(L(X,P))\}, 
\end{equation}
where $j\in J(P)$ is an arbitrary job of maximal length.
\end{lemma}
\begin{lemma}\label{lem:interesting_spans}
  While computing $c(\mathcal P)$ using recursion~\eqref{eq:unidir_recursion}, if $c(P)$ is computed
  for some $P\neq\emptyset$ then $P=[\ell_a,\ell_b]$ with
  \begin{align*}
    a &\in T_s:=\{1\}\cup\left\{r'_j+p'_j\,:\,j\in\mathcal
        J\right\}\cup\left\{d'_{j'}-p'_{j'}+p'_j+1\,:\,j,j'\in\mathcal J\right\}, \\
    b &\in T_e:=\{m\}\cup\left\{d'_j-p'_j\,:\,j\in\mathcal
        J\right\}\cup\left\{r'_j-1\,:\,j\in\mathcal
        J\right\}
  \end{align*}
\end{lemma}
\begin{proof}
  We proceed by induction on the cardinality of the set $P$. The base case
  $\mc{P} = [\ell_1, \ell_m]$ is trivial. Let $0<\lvert P\rvert<m$ and assume the statement is
  satisfied for all $P'$ with $\lvert P'\rvert>\lvert P\rvert$. We have $P = R(X,P')$ or
  $P = L(X,P')$ for some $P' \subseteq \mc{P}$ with $\lvert P'\rvert>\lvert P\rvert$ and some
  $X \in \mc{R}_j^e$ for a job $j \in J(P')$. By induction,
  $P'=[\ell_{\alpha}, \ell_{\beta}]$ with $\alpha\in T_s$ and $\beta\in T_e$. From $X \in
  \mc{R}_j^e$ it follows that $X = [\ell_i, \ell_k]$ with
  \begin{align*}
    i&\in \{r'_j\}\cup\{d'_{j'}-p'_{j'}+1\,:\,j'\in\mathcal J\}\},&
    k&\in \{r'_j+p'_j-1\}\cup\{d'_{j'}-p'_{j'}+p'_j\,:\,j'\in\mathcal J\}\}.
  \end{align*}
  \begin{description}
  \item[Case 1] If $i\leq \alpha\leq k<\beta$, then  $L(X,P') = \emptyset$ and $R(X,P') =[\ell_{k+1},\ell_\beta]$.
  \item[Case 2] If $i\leq \alpha\leq\beta\leq k$, then $L(X,P') = R(X,P') = \emptyset$.
  \item[Case 3] If $\alpha<i\leq k<\beta$, then $L(X,P') = [\ell_\alpha,\ell_{i-1}]$ and $R(X,P') =[\ell_{k+1},\ell_\beta]$.
  \item[Case 4] If $\alpha<i\leq \beta\leq k$, then $L(X,P') = [\ell_\alpha,\ell_{i-1}]$ and $R(X,P') =\emptyset$.
  \end{description}
  In every case, the claim follows from $k+1\in T_s$ and $i-1\in T_e$.
\end{proof}
\begin{theorem}\label{theorem:runtime}
  The unidirectional maintenance scheduling problem with $n$ jobs can be solved in time $O(n^4)$.
\end{theorem}
\begin{proof}
  By Lemma~\ref{lem:interesting_spans}, the number of subproblems in the dynamic program for the
  recursion~\eqref{eq:unidir_recursion}, is $O(\lvert T_s\rvert\lvert T_e\rvert)=O(n^3)$. In each
  subproblem, after picking the splitting job $j$, we have $\lvert \mathcal R^e_j\rvert=O(n)$
  choices for the possession $X$, which implies a runtime bound of $O(n^4)$ (see \cite[Section 15.3]{CormenSRL2009}).  
\end{proof}
In Section~\ref{subsec:ordered_jobs} we will present another algorithm with a runtime of $O(n^2)$ under
an additional assumption on the instance.

\subsection{The integer programming model}\label{subsec:IP}
For the unidirectional maintenance scheduling problem, the graph $G$ whose adjacency relation is
specified in Definition~\ref{def:adjacent} is a path, and the set covering model can be simplified
as follows. The set $F_0$ of spans is $\{[\ell_i,\ell_k]\,:\,1\leq i\leq k\leq m\}$, where
$[\ell_i,\ell_k]$ denotes the set $\{\ell_i,\ell_{i+1},\dots,\ell_k\}$. Let $x_{ik}$ be a binary
variable indicating that the paths in this span are cancelled. For $j\in\mathcal J$, let
$B_j=\{(i,k)\,:\,\exists R\in\mathcal R_j\,R\subseteq[\ell_i,\ell_k]\}$. Then the unidirectional set
cover model (uniSCM) is
\begin{align}
  \text{[uniSCM] }\text{Minimise} \sum_{1\leq i\leq k\leq m}(k-i+1)&x_{ik}\quad  \text{subject to}\nonumber\\
  \sum_{(i,k)\in B_j}x_{ik} &\geq 1&&\text{for all } j\in \mc{J},   \label{eq:uniSCM_covering}\\
  x_{ik} &\in\{0,1\} && \text{for all } (i,k)\text{ with }1\leq i\leq k\leq m.\label{eq:uniSCM_binaries}
\end{align}
This is a binary program with $\binom{m+1}{2}$ variables and $n$ constraints. 

For the case that all jobs have the same length the problem can be solved efficiently by linear programming. This is a
consequence of the following lemma.
\begin{lemma}\label{lem:unimodularity_of_SCM}
  For an instance of the unidirectional maintenance scheduling problem with $p'_j=p'$ for all
  $j \in \mc{J}$, the LP relaxation of the problem [uniSCM] is integral.
\end{lemma}
\begin{proof}
  By~\eqref{eq:nondom_equal}, we can assume without loss of generality, $r'_1<r'_2<\dotsb<r'_{\lvert\mc{J}\rvert}$ and
  $d'_1<d'_2<\dotsb<d'_{\lvert\mc{J}\rvert}$.  Let $A$ be the constraint matrix
  for~\eqref{eq:uniSCM_covering}. We will verify that $A$ is an interval matrix, that is, in every
  column the 1's appear consecutively. Then $A$ is totally unimodular (see, e.g.,~\cite[Corollary
  III.2.10]{LaurenceA.Wolsey1999}, and the result follows (see, e.g., \cite[Proposition
  III.2.2]{LaurenceA.Wolsey1999}). We denote the value of row $j \in \mc{J}$ and column $(i,k)$ in
  matrix $A$ by $a_{jik}$:
  \[a_{jik}=
    \begin{cases}
      1 &\text{if }\left\lvert[i,k]\cap [r'_j,d'_j]\right\rvert \geq p'\text{ (that is, if $(i,k) \in B_j$)},\\
      0 &\text{otherwise.}
    \end{cases}
  \]
  Consider three row indices
  $j_1 < j_2 < j_3 \in \mc{J}$ with $a_{j_1ik}=a_{j_3ik}=1$.  We need to check that this implies
  $a_{j_2ik}=1$. From $a_{j_1ik}=1$, we have $\lvert [i,k] \cap [r'_{j_1},d'_{j_1}]\rvert\geq p'$,
  hence $i\leq d'_{j_1}-p'+1$ and, using $d'_{j_1} < d'_{j_2}$, we obtain 
\begin{equation}\label{eq:unimodularity_of_SCM_1}
  i \leq d'_{j_2} - p' +1. 
\end{equation}
From $a_{j_3ik}=1$, we have $\lvert [i,k] \cap [r'_{j_3},d'_{j_3}]\rvert\geq p'$,
  hence $k\geq r'_{j_3}+p'-1$ and, using $r'_{j_2} < r'_{j_3}$, we obtain 
\begin{equation}\label{eq:unimodularity_of_SCM_2}
  k \leq r'_{j_2} + p' -1. 
\end{equation}
From~(\ref{eq:unimodularity_of_SCM_1}) and~(\ref{eq:unimodularity_of_SCM_2}), it follows that
$\lvert [i,k] \cap [r'_{j_2},d'_{j_2}]\rvert\geq p'$, and therefore $a_{j_2ik}=1$, as required. 
\end{proof}	
The number of variables in both [PIM] and [uniSCM] can be reduced by restricting the sets of paths
that can be the first or last path of a span or possession. 
As a consequence, we can restrict [uniSCM] to the variables $x_{ik}$ with $i\in A$ and $k\in B_i\cup
C$, and [PIM] to the variables $X_R$ where $R\in\mathcal R_j$ has the form $[\ell_i,\ell_k]$ with
$i\in A\cup\{\ell_{e'_j}\}$. This leads to versions of [uniSCM] and [PIM] with $O(n^2)$ variables
and $n$ constraints, and $O(n^2+m)$ variables and $O(n^2m)$ constraints, respectively.
We conclude this subsection by comparing the strength of the LP relaxations for [uniSCM] and [PIM].
\begin{proposition}\label{prop:comparison_SCM_PIM}
  For the unidirectional maintenance problem, the LP relaxation of the model [uniSCM] is at least
  as strong as the LP relaxation of the model [PIM].
\end{proposition}
\begin{proof}
  Let $z_1$ and $z_2$ be the optimal values of the LP relaxations of [uniSCM] and [PIM],
  respectively. We have to show that $z_1 \geq z_2$. Let $x=(x_{ik})$ be an optimal solution for the
  LP relaxation of [uniSCM]. We define a solution for LP relaxation of [PIM] by
  \begin{align}
    y_{\ell_r} &= \sum_{(i,k)\,:\,i\leq r\leq k}x_{ik} \text{ for all }r\in[1,m],&
    x_R &= \min\{y_{\ell}\,:\,\ell\in R\}  \text{ for all }R\in\mathcal R.\label{eq:PIM_sol}
  \end{align}
Constraint~\eqref{eq:PIM_2} is satisfied by definition. For constraint~\eqref{eq:PIM_1}, fix
$j\in\mathcal J$, and for every $R \in \mc{R}_j$ an index $r(R)$ with $\ell_{r(R)}\in R$ and $x_R= y_{\ell_{r(R)}}$. Then
\begin{multline*}
  \sum_{R \in \mc{R}_j}x_R = \sum_{R \in \mc{R}_j}y_{\ell_{r(R)}} = \sum_{R \in
    \mc{R}_j}\sum_{(i,k)\,:\,i\leq r(R)\leq k}x_{ik} \geq \sum_{R \in \mc{R}_j}\sum_{(i,k)\,:\,R
    \subseteq [\ell_i,\ell_k]}x_{ik}\\
  =\sum_{1\leq i\leq k\leq m}\sum_{R \in \mc{R}_j: R \subseteq [\ell_i,\ell_k]}x_{ik}\geq
  \sum_{(i,k) \in \mc{F}: \exists R \in \mc{R}_j, R \subseteq [\ell_i, \ell_k]}x_{ik} \geq 1.
\end{multline*}
So~\eqref{eq:PIM_sol} defines a feasible solution for the LP relaxation of [PIM], and we conclude
the proof as follows:
\[ z_2\leq\sum_{r \in [1,m]}y_{\ell_r} = \sum_{r \in [1,m]}\sum_{(i,k)\,:\,i\leq r\leq k}x_{ik}  
                               = \sum_{1\leq i\leq k\leq m}\sum_{r \in [1,m]\,:\, i\leq r\leq
                                 k}x_{ik}=\sum_{1\leq i\leq k\leq m}(k-i+1)x_{ik}=z_1.\qedhere\]
\end{proof}
We provide an example in which [uniSCM] is strictly stronger than [PIM].
\begin{example}
  Consider an instance of the problem with three jobs as follows:
  \begin{align*}
    \mc{R}_1&=\left\{[\ell_2,\ell_3],[\ell_3,\ell_4]\right\}, &
                                                                \mc{R}_2&=\left\{[\ell_3,\ell_5],[\ell_4,\ell_6],[\ell_5,\ell_7]\right\}, &
                                                                \mc{R}_3&=\left\{[\ell_6,\ell_7],[\ell_7,\ell_8]\right\}.
  \end{align*}
  An optimal solution of the LP-relaxation of [PIM] is
  \begin{align*}
    (y_{\ell_2},y_{\ell_3},y_{\ell_4},y_{\ell_5},y_{\ell_6},y_{\ell_7},y_{\ell_8})&=(1/2,\,1/2,\,1/2,\,1/3,\,1/2,\,1/2,\,1/2), \\
    x_{[\ell_2,\ell_3]}=x_{[\ell_3,\ell_4]}=x_{[\ell_6,\ell_7]}=x_{[\ell_7,\ell_8]} &=1/2, \\
    x_{[\ell_3,\ell_5]}=x_{[\ell_4,\ell_6]}=x_{[\ell_5,\ell_7]}&=1/3,
  \end{align*} 
  with value $ 10/3 $, and for [uniSCM], it is
  \[ x_{[\ell_3,\ell_4]}=x_{[\ell_5,\ell_7]}=1\]
  with value 5 which is the optimal solution of the instance.
\end{example}
In Section~\ref{sec:computation}, we provide computational evidence consistent with the
theoretical result of Proposition~\ref{prop:comparison_SCM_PIM}, but shows that the performance of
[uniSCM] can be much better. 

\subsection{A special case}\label{subsec:ordered_jobs}
The runtime bound of $O(n^4)$ for the algorithm described in
Section~\ref{subsec:unidir_interval_poly_time_algorithm} can be improved to $O(n^2)$ when the
deadlines are ordered the same way as the release times, that is,
\begin{align}
  &r'_1 \leq r'_2 \leq \cdots \leq r'_n,  &&  d'_1 \leq d'_2 \leq \cdots \leq d'_n. \label{eq:ordering} 
\end{align} 
Under this assumption the unidirectional maintenance scheduling problem can be reduced to a
shortest path problem as follows. We define a directed graph $D=(V,A)$ with node set $V=\{0,1,2,\dots,n\}$, and
arc set $A=\{(v,w)\,:\,0\leq v<w\leq n\}$. The weight of arc $(v,w)$ is defined by
$\wt(vw)=\beta(v,w)-\alpha(v)+1$, where
\begin{align}
  \alpha(v) &=\min\left\{d'_j-p_j+1\,:\,v+1\leq j\leq n\in[n]\right\} && 0\leq v\leq n-1,\label{eq:alpha}\\
  \beta(v,w) &=\max\left\{\max\{r'_j,\alpha(v)\}+p_j-1\,:\,v+1\leq j\leq w\right\} &&0\leq v<w\leq n.\label{eq:beta}                                                                        
\end{align}
We will show that a shortest path from $0$ to $n$ in the digraph $D$ corresponds to an optimal
solution of the maintenance scheduling problem. The basic idea is that using arc $(v,w)$ in the path
from $0$ to $n$ corresponds to cancelling the paths with indices $\alpha(v),\dots,\beta(v,w)$, which
allows the scheduling of the jobs $v+1,v+2,\dots,w$. We start by showing that in this way every path
corresponds to a feasible schedule, and consequently the minimum length of a path is an upper bound
on the objective value for the maintenance scheduling problem.
\begin{lemma}\label{lem:upper_bound}
  Let $P=(0=v_0,v_1,v_2,\dots,v_t=n)$ be a path in $D$, and set
  \begin{align*}
    x_{ik}&=
    \begin{cases}
      1 & \text{if }(i,k)=(\alpha(v_s),\beta(v_s,v_{s+1}))\text{ for some }s\in\{0,1,\dots,t-1\}\\
      0 & \text{otherwise}.
    \end{cases} &&\text{for all }[\ell_i,\ell_k]\in F_0.
  \end{align*}
  Then $(x_{ik})_{1\leq i\leq k\leq m}$ satisfies~\eqref{eq:uniSCM_covering}, and
  $\displaystyle\sum_{1\leq i\leq k\leq m}(k-i+1)x_{ik}=\wt(P)$.
\end{lemma}
\begin{proof}
  Let $j\in[m]$ be an arbitrary job, let $s\in\{0,1,\dots,t-1\}$ be the unique index with
  $v_s<j\leq v_{s+1}$, and set $i=\alpha(v_s)$, $k=\beta(v_s,v_{s+1})$. Then $x_{ik}=1$ by
  assumption, and in order to prove that~\eqref{eq:uniSCM_covering} is satisfied, we verify
  $(i,k)\in B_j$. For this purpose set $a=\max\{i,r'_j\}$, $b=a+p_j-1$. Then
  $\lvert[a,b]\rvert=b-a+1=p_j$, and we claim that
  $[a,b]\subseteq [i,k]\cap[r'_j,d'_j]=[\max\{i,r'_j\},\,\min\{k,d'_j\}]$, which then implies
  $(i,k)\in B_j$. To see the claim, we first note that $a\geq\max\{i,r'_j\}$ by definition. Then
  $b\leq k$ follows from~\eqref{eq:beta} for $(v,w)=(v_s,v_{s+1})$, together with
  $v_s+1\leq j\leq v_{s+1}$:
  \[ k\geq\max\{r'_j,i\}+p_j-1=a+p_j-1=b.\]
  For $b\leq d'_j$, we note that $r'_j+p_j-1\leq d'_j$, since we assume that the instance is
  feasible, and $i+p_j-1\leq d'_j$ follows from~\eqref{eq:alpha} for $v=v_s$, together with
  $v_s+1\leq j\leq n$: $i=\alpha(v_s)\leq d'_j-p_j+1$.

  We have proved that $(x_{ik})$ satisfies~\eqref{eq:uniSCM_covering}, and it remains to evaluate the
  objective function:
  \begin{multline*}
    \sum_{1\leq i\leq k\leq m}(k-i+1)x_{ik}=\sum_{(i,k)\in
      \{(\alpha(v_s),\beta(v_s,v_{s+1}))\,:\,s=0,1,\dots,t-1\}}(k-i+1)\\
    =\sum_{s=0}^{t-1}\left(\beta(v_s,v_{s+1})-\alpha(v(s))+1\right)=\sum_{s=0}^{t-1}\wt(v_s,v_{s+1})=\wt(P).\qedhere
  \end{multline*}
\end{proof}

In order to show that a shortest path corresponds to an optimal schedule, let
$(x_{ik})_{1\leq i\leq k\leq m}$ be an optimal solution with minimum support, that is $\sum_{(i,k)}x_{ik}$
is minimal among all optimal solutions. Let $S=\{(i,k)\,:\,x_{ik}=1\}$ be the support of this
solution. The minimality assumption implies that we can label the elements of $S$ as
$(i_1,k_1)$,\dots,$(i_t,k_t)$, such that $i_{s+1}\geq k_s+2$ for all $s\in\{1,2,\dots,t-1\}$. We
define a function $f:\{1,\dots,t\}\to\{0,\dots,n\}$, where $f(s)=0$ if job $1$ cannot be scheduled into the first $s$
spans, and otherwise $f(s)$ is the maximal job index $j$ such that all jobs up to $j$ can be
scheduled into the first $s$ spans. More formally,
\[f(s) = \max\{j\in[n]\,:\,\forall j'\in\{1,\dots,j\}\,\exists s'\in\{1,\dots,s\}\
  \left\lvert[i_{s'},k_{s'}]\cap[r'_{j'},d'_{j'}]\right\rvert\geq p_{j'}\},\]
where we define the maximum of the empty set to be zero. Since $(x_{ik})$ is a feasible solution, we
have $f(t)=n$, and the the definition of $f$ immediately implies monotonicity:
\[0\leq f(1)\leq f(2)\leq\dots\leq f(t)=n.\]
In the next lemma we use the minimality assumption on the solution $(x_{ik})$ to establish strict monotonicity, that is, each of the inequalities in the
chain above is strict. The intuition behind the formal proof provided below is as follows. If
$f(s)=f(s-1)=j^*$ then jobs $1$ to $j^*$ can be scheduled in spans $1$ to $s-1$, but job $j^*+1$
can't be scheduled in any of the first $s$ spans. As $(x_{ik})$ is a feasible solution, job $j^*+1$
can be scheduled into a span $[i_{s^*},k_{s^*}]$ for some $s^*\geq s+1$. Then the
assumption~\eqref{eq:ordering} can be used to
show that the spans $s$ to $s^*-1$ can be omitted which contradicts the assumption that $(x_{ik})$
is an optimal solution.
\begin{lemma}\label{lem:strong_monotonicity}
  $0<f(1)<f(2)<\dots<f(t)=n$.
\end{lemma}
\begin{proof}
  Suppose the statement is false, and let $s$ be the first argument where a violation occurs, that
  is, $s=1$ if $f(1)=0$, and otherwise $s=\min\{s'\,:\,f(s')=f(s'-1)\}$. Let $j^*=f(s)$. By
  definition of $f$,
  \[\left\lvert[r'_{j^*+1},d'_{j^*+1}]\cap[i_{s'},k_{s'}]\right\rvert<p_{j^*+1}\qquad\text{ for all }s'\leq s,\]
  and by feasibility of $(x_{ik})$, there is a smallest $s^*\in\{s+1,\dots,t\}$ with
  \[\left\lvert[r'_{j^*+1},d'_{j^*+1}]\cap[i_{s^*},k_{s^*}]\right\rvert\geq p_{j^*+1}.\]
  We claim that the vector $(x'_{ik})$ defined by $x'_{ik}=0$ if
  $(i,k)=\left(i_{s},k_{s}\right)$, and $x'_{ik}=x_{ik}$ otherwise, is still a feasible
  solution, and this is the required contradiction to the assumption that $(x_{ik})$ is an optimal
  solution. To show the claim assume that it is false. This implies that there exists a job $j$ such
  that
  \begin{equation}\label{eq:assumption_j}
  \left\lvert[r'_j,d'_j]\cap[i_{s'},k_{s'}]\right\rvert<p_j\qquad\text{ for all }
  s'\in\{1,\dots,t\}\setminus\{s\}.  
\end{equation}
The jobs up to job $j^*$ can be scheduled into the first $s-1$ spans, and job $j^*+1$ can be
scheduled into span $s^*$, hence $j\geq j^*+2$. Feasibility of $(x_{ik})$ implies
$\left\lvert[r'_j,d'_j]\cap[i_{s},k_{s}]\right\rvert\geq p_j$, and in particular
\begin{equation}\label{eq:condition_1}
  r_j\leq k_{s}<i_{s^*}.   
\end{equation}
If $p_j\geq p_{j^*+1}$, then using $(i_{s},k_{s})\in B_j$ and $r_{j^*+1}\leq r_j$ we obtain,
\[\max\{i_{s},\,r_{j^*+1}\}+p_{j^*+1}-1\leq\max\{i_{s},\,r_{j}\}+p_j-1\leq k_{s},\]
which implies $(i_{s},k_{s})\in B_{j^*+1}$, contradicting our assumption $f(s)=j^*$. We conclude
$p_j<p_{j^*+1}$, and then
\begin{equation}\label{eq:condition_2}
  \max\{i_{s^*},r_j\}+p_j-1\stackrel{\eqref{eq:condition_1}}{=}i_{s^*}+p_j-1 <i_{s^*}+p_{j^*+1}-1\leq k_{s^*},
\end{equation}  
where the last inequality comes from $(i_{s^*},k_{s^*})\in B_{j^*+1}$. Now~\eqref{eq:condition_2}
implies $(i_{s^*},k_{s^*})\in B_{j}$ which contradicts our assumption~\eqref{eq:assumption_j}, and
this concludes the proof.
\end{proof}

\begin{lemma}\label{lem:interval_partition}
  For every $s\in\{2,\dots,t\}$ and every $j\in\{f(s-1)+1,\,f(s-1)+2,\dots,\,f(s)\}$,
  $(i_s,k_s)\in B_j$.
\end{lemma}
\begin{proof}
  Fix $s\in\{2,\dots,t\}$, $j\in\{f(s-1)+1,\,f(s-1)+2,\dots,\,f(s)\}$, and set $j^*=f(s-1)+1$. By
  Lemma~\ref{lem:strong_monotonicity}, $(i_s,k_s)\in B_{j^*}$ but $(i_{s'},k_{s'})\not\in B_{j^*}$
  for all $s'<s$. By definition of $f$, $(i_{s'},k_{s'})\in B_j$ for some $s'\leq s$. If $s'=s$ then
  there is nothing to do, so assume $s'<s$. From
  $\left\lvert[r'_j,d'_j]\cap[i_{s'},k_{s'}]\right\rvert\geq p_j$ it follows that
  \begin{equation}\label{eq:condition_1a}
   r_j\leq k_{s'}<i_{s}.   
  \end{equation}    
  If $p_j\geq p_{j^*}$, then using $(i_{s'},k_{s'})\in B_{j}$ and $r_{j^*}\leq r_j$ we obtain, 
  \[\max\{i_{s'},\,r_{j^*}\}+p_{j^*}-1\leq\max\{i_{s'},\,r_{j}\}+p_j-1\leq k_{s'},\]
  which implies $(i_{s'},k_{s'})\in B_{j^*}$, contradicting our assumption $f(s-1)=j^*-1$. We conclude
  $p_j<p_{j^*}$, and then
  \begin{equation}\label{eq:condition_2a}
    \max\{i_{s},r_j\}+p_j-1\stackrel{\eqref{eq:condition_1a}}{=}i_{s}+p_j-1 <i_{s}+p_{j^*}-1\leq k_{s},
  \end{equation}  
  where the last inequality comes from $(i_{s},k_{s})\in B_{j^*}$. Now~\eqref{eq:condition_2a}
  implies $(i_{s},k_{s})\in B_{j}$, as required.  
\end{proof}
\begin{lemma}\label{lem:lower_bound}
  Let $P$ be the path $(0,f(1),f(2),\dots,f(t)=n)$ in the digraph $D$. Then
  \[\sum_{(i,k)\in F_0}(k-i+1)x_{ik}\geq\wt(P).\]
\end{lemma}
\begin{proof}
  The statement follows from
  \begin{equation}\label{eq:arc_weights}
    k_s-i_s+1 \geq\wt(f(s-1),f(s))= \beta(f(s-1),f(s))-\alpha(f(s-1))-1 \text{ for all }s\in\{1,2,\dots,t\},
  \end{equation}
  where we extend $f$ by setting $f(0)=0$. In order to verify~\eqref{eq:arc_weights}, we assume
  without loss of generality, that replacing $(i_s,k_s)$ with $(i_s+1,k_s+1)$ does not lead to a
  feasible solution. This implies that $i_s=d'_j-p'_j+1$ for some
  $j\in\{f(s-1)+1,\dots,n\}$. Furthermore, since $i_s<i_{s+1}<\dots<i_t$, and by
  Lemma~\ref{lem:interval_partition} all the jobs $j\in\{f(s-1)+1,\dots,n\}$ can be scheduled in one
  of the spans $[i_s,k_s]$,\dots $[i_t,k_t]$, we conclude
  \[i_s=\min\left\{d'_j-p_j+1\,:\,j=f(s-1)+1,\dots,n\right\}=\alpha(f(s-1)).\]
  From $(i_s,k_s)\in B_j$ for all $j\in\{f(s-1)+1,\dots,s\}$, it follows that
  \[k_s\geq\max\left\{\max\{r'_j,\alpha(v)\}+p_j-1\,:\,f(s-1)+1\leq j\leq
      f(s)\right\}=\beta(f(s-1),f(s)),\]
  and since $(x_{ik})$ is optimal, we obtain $k_s=\beta(f(s-1),f(s))$, and this concludes the proof.
\end{proof}
Combining Lemmas~\ref{lem:upper_bound} and~\ref{lem:lower_bound}, we obtain the following theorem.
\begin{theorem}\label{thm:quadratic_algorithm}
  Under assumption~\eqref{eq:ordering}, any shortest path from $0$ to $n$ in the digraph $D$
  corresponds to an optimal solution of the unidirectional maintenance scheduling problem. In
  particular, the problem can be solved in time $O(n^2)$.
\end{theorem}

\section{The bidirectional case with singleton possessions and a bounded number of intersections}\label{sec:bounded_intersections}
In this section we make the following assumptions.
\begin{assumption}\label{ass:singleton_possessions}
  Each possession contains exactly one path.   
\end{assumption}
\begin{assumption}\label{ass:bounded_intersections}
  The number of intersections of each path with other paths is bounded.
\end{assumption}
We say that an instance satisfying the first condition is an instance with \emph{singleton
  possessions}. Since the path starting at time $i\Delta$ intersects precisely the paths in the
opposite direction which arrive in the time interval $[i\Delta,i\Delta+2\delta]$ (see
Figure~\ref{fig:intersection_up_down}), every path intersects at most
$\lfloor2\delta/\Delta\rfloor+1$ other paths, and the second condition is equivalent to the
requirement that the ratio $\delta/\Delta$ is bounded above.
\begin{figure}
  \centering
  \begin{tikzpicture}[xscale=.8]
    \draw[thick,->] (-.5,0) -- (16,0);
    \draw[thick,->] (0,-.5) -- (0,5.5);
    \foreach \x in {1,...,12} \draw (\x-2,5) -- (\x+3,0);
    \draw[very thick] (4,0) -- (9,5);
    \node at (9,5.3) {$\ell_i^u$};
    \node at (6,5.3) {$\ell_{i'}^d$};
    \draw[<->] (7,5.3) to node[above] {$\Delta$} (8,5.3);
    \draw (7,5.1) -- (7,5.5);
    \draw (8,5.1) -- (8,5.5);
    \node at (4,-.5) {{\small$i\Delta$}};
    \node at (6,-.5) {{\small$i'\Delta$}};
    \node at (9,-.5) {{\small$i\Delta+\delta$}};
    \node at (11,-.5) {{\small$i'\Delta+\delta$}};
    \draw[<->] (4,-1.3) to node[above] {$2\delta$} (14,-1.3);
    \draw (4,-1.1) -- (4,-1.5);
    \draw (14,-1.1) -- (14,-1.5);
    \draw[dashed] (6,-.3) -- (6,5);
    \draw[dashed] (4,0) -- (4,-.3);
    \draw[dashed] (9,5) -- (9,-.3);
    \draw[dashed] (11,0) -- (11,-.3);    
  \end{tikzpicture}
  \caption{Intersection of an up-path with $11$ down-paths. The path $\ell_{i'}^d$ intersects
    $\ell_i^u$ if and only if $i\Delta\leq i'\Delta+\delta \leq i\Delta+2\delta$.}\label{fig:intersection_up_down}
\end{figure}
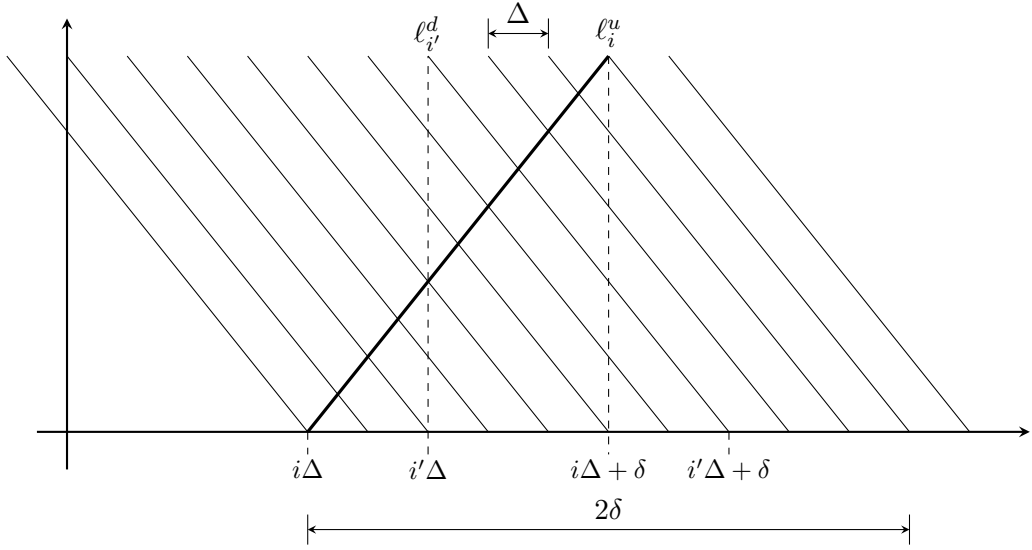
Extending the shortest path approach from Section~\ref{subsec:ordered_jobs}, we will show that under
these two assumptions the maintenance scheduling problem can be solved in polynomial time.

Given Assumption~\ref{ass:singleton_possessions}, we can identify $\mathcal R_j$ with the set
$\{\ell\in\mathcal P\,:\,\{\ell\}\in\mathcal R_j\}$, so that a solution for a job set
$J \subseteq \mc{J}$ is a set $X\subseteq\mathcal P$ such that $X\cap\mathcal R_j\neq\emptyset$ for
every job $j\in J$. 
\begin{example}\label{ex:instance}
  For the instance shown in Figure~\ref{fig:assumptions}, we have
  \begin{align*}
    \mathcal R_1 &= \{\ell_1^d,\,\ell_4^u\},&
    \mathcal R_2 &= \{\ell_2^d,\,\ell_4^u,\ell_3^d,\ell_5^u\},\\
    \mathcal R_3 &= \{\ell_3^d,\,\ell_1^u,\ell_4^d,\ell_2^u,\ell_5^d,\,\ell_3^u\},&
    \mathcal R_4 &= \{\ell_5^d,\,\ell_5^u,\ell_6^d\},\\
    \mathcal R_5 &= \{\ell_3^u,\,\ell_7^d,\ell_4^u,\,\ell_8^d,\,\ell_5^u,\ell_9^d\},&
    \mathcal R_6 &= \{\ell_8^u,\,\ell_6^d,\ell_9^u,\,\ell_7^d,\,\ell_{10}^u,\ell_8^d\},\\
    \mathcal R_7 &= \{\ell_{10}^d,\,\ell_8^u,\ell_{11}^d,\,\ell_9^u\},&
    \mathcal R_8 &= \{\ell_{10}^d,\,\ell_{10}^u,\ell_{11}^d,\,\ell_{11}^u\}
  \end{align*}
  and a feasible solution is given by $X=\{\ell_3^d,\,\ell_4^u,\,\ell_5^u,\,\ell_9^u,\,\ell_{11}^d\}$.
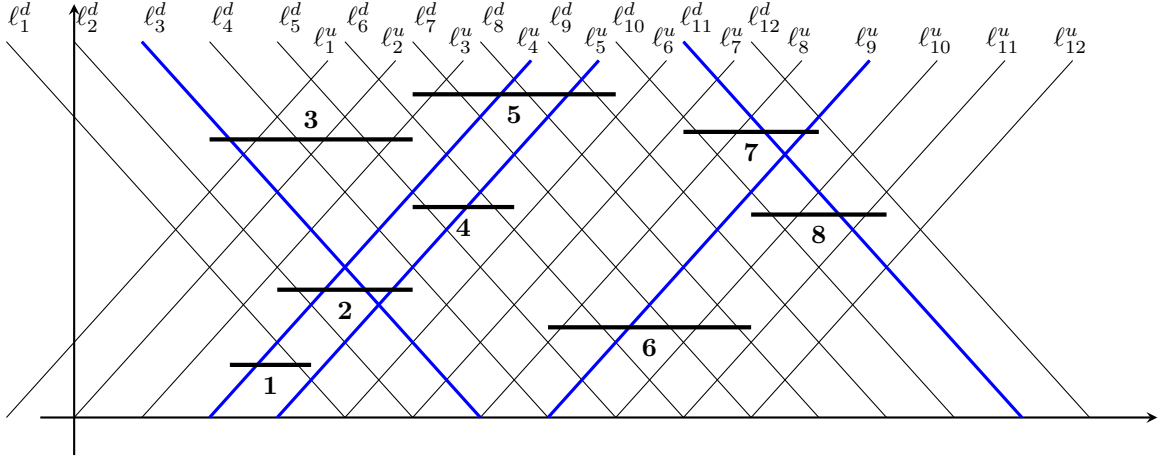
\begin{figure}
  \centering
  \begin{tikzpicture}[xscale=.9]
    \draw[thick,->] (-.5,0) -- (16,0);
    \draw[thick,->] (0,-.5) -- (0,5.5);
    \foreach \x in {1,...,12} {
      \draw (\x-2,5) -- (\x+3,0);
      \draw (\x-2,0) -- (\x+2.75,4.75);
      \draw (\x-1.8,5) node[anchor=south] {{\small $\ell_{\x}^d$}};
      \draw (\x+2.72,4.72) node[anchor=south] {{\small $\ell_{\x}^u$}};
    }
    \draw[very thick,blue] (1,5) -- (6,0);
    \draw[very thick,blue] (9,5) -- (14,0);
    \draw[very thick,blue] (2,0) -- (6.75,4.75);
    \draw[very thick,blue] (3,0) -- (7.75,4.75);
    \draw[very thick,blue] (7,0) -- (11.75,4.75);
    \draw[ultra thick] (2,3.7) to node[above] {{\small$\mathbf{3}$}} (5,3.7);
    \draw[ultra thick] (3,1.7) to node[below] {{\small$\mathbf{2}$}} (5,1.7);
    \draw[ultra thick] (5,2.8) to node[below] {{\small$\mathbf{4}$}} (6.5,2.8);
    \draw[ultra thick] (5,4.3) to node[below] {{\small$\mathbf{5}$}} (8,4.3);
    \draw[ultra thick] (7,1.2) to node[below] {{\small$\mathbf{6}$}} (10,1.2);
    \draw[ultra thick] (10,2.7) to node[below] {{\small$\mathbf{8}$}} (12,2.7);
    \draw[ultra thick] (9,3.8) to node[below] {{\small$\mathbf{7}$}} (11,3.8);
    \draw[ultra thick] (2.3,.7) to node[below] {{\small $\mathbf 1$}} (3.5,.7);    
  \end{tikzpicture}
  \caption{An instance of the maintenance scheduling problem as considered in this section. The
    eight horizontal lines represent the jobs, and a feasible solution is indicated in blue.}\label{fig:assumptions}
\end{figure}
\end{example}
The chronological order in which paths intersect with job $j$ induces an order on the set
$\mathcal{R}_j$, and in the following we will refer to this order when we talk about the first, last
or $k$-th path of $\mathcal R_j$. We are looking for a solution $X$ for $\mathcal J$ with minimum
cardinality $\lvert X\rvert$, and we will prove that such an $X$ can be determined in time
polynomial in $n=\lvert \mc{J}\rvert$ as long as $\delta/\Delta=O(1)$.

For every $j$, let $\mc{R}^*_j$ be the set containing the last up-path (if it exists) and the last
down-path (if it exists) in $\mathcal R_j$. So $\lvert\mathcal R_j^*\rvert=2$ if
$\mathcal R_j$ contains both up and down-paths, and $\lvert \mathcal R_j^*\rvert=1$, otherwise.  Set
$\mathcal P^*=\bigcup_{j\in\mathcal J}\mc{R}^*_j$.
\begin{example}
  The reduced possession sets for the instance in Figure~\ref{fig:assumptions} are
  \begin{align*}
    \mathcal R^*_1 &= \{\ell_1^d,\,\ell_4^u\},& \mathcal R^*_2 &= \{\ell_3^d,\ell_5^u\},&
    \mathcal R^*_3 &= \{\ell_5^d,\,\ell_3^u\},& \mathcal R^*_4 &= \{\ell_5^u,\ell_6^d\},\\
    \mathcal R^*_5 &= \{\ell_5^u,\ell_9^d\},& \mathcal R^*_6 &= \{\ell_{10}^u,\ell_8^d\},&
    \mathcal R^*_7 &= \{\ell_{11}^d,\,\ell_9^u\},& \mathcal R^*_8 &= \{\ell_{11}^d,\,\ell_{11}^u\},
  \end{align*}
  so that $\mathcal P^*=\{\ell_1^d,\,\ell_3^d,\,\ell_5^d,\,\ell_{6}^d,\,\ell_{8}^d,\,\ell_{9}^d,\,\ell_{11}^d,\,\ell_{3}^u,\,\ell_{4}^u,\,\ell_{5}^u,\,\ell_{9}^u,\,\ell_{10}^u,\,\ell_{11}^u\}$.
\end{example}
\begin{lemma}\label{lem:interesting_paths}
  For any job set $\mc{J}$, there exists an optimal solution $X\subseteq\mathcal P^*$.
\end{lemma}
\begin{proof}
  Suppose the statement is false. Let $X$ be an optimal
  solution with $\lvert X\setminus\mathcal P^*\rvert$ minimal. First assume
  $\ell^u_i\in X\setminus\mathcal P^*$, set $X'=X\setminus\{\ell^u_{i}\}$, and let
  $J'=\{j\in \mc{J}\,:\,\mc{R}_j\cap X'=\emptyset\}$. Then feasibility of $X$ implies $\ell^u_i\in \mc{R}_j$ for
  all $j\in J'$, and $\ell^u_i\not\in\mathcal P^*$ implies $\ell^u_{i+1}\in \mc{R}_j$ for all $j\in
  J'$. With $i'=\max\{i''\,:\,\ell^u_{i''}\in \mc{R}_j\text{ for all }j\in J'\}$, we have
  $\ell^u_{i'}\in\mathcal P^*$ (otherwise $\ell^u_{i'+1}\in \mc{R}_j$ for all $j\in J'$), and then
  $X''=X'\cup\{\ell^u_{i'}\}$ is an optimal solution with
  $\lvert X''\setminus\mathcal P^*\rvert<\lvert X\setminus\mathcal P^*\rvert$, contradicting the
  minimality assumption on $\lvert X\backslash \mc{P}^* \rvert$. The argument for $\ell^d_i\in X\setminus\mathcal P^*$ is similar.
\end{proof}
\begin{definition}\label{def:block}
  For $B \subseteq \mc{P}$, a pair $(\ell,B)$ with $\ell \in B$ is called a \emph{block} if
  $\ell' \cap \ell \neq \emptyset$ for all $\ell'\in B$. Furthermore, the pair $(0,\emptyset)$ is
  called a \emph{null block}. 
\end{definition}
\begin{definition}\label{def:X_l_block}
  For a path set $X \subseteq \mc{P}$ and a path $\ell\in X$, we define a block $B(\ell,X)$ by
  $B(\ell,X)=\{\ell'\in X\,:\,\ell\cap\ell'\neq\emptyset\}$.
\end{definition}
\begin{example}
  For the solution $X=\{\ell_3^d,\,\ell_4^u,\,\ell_5^u,\,\ell_9^u,\,\ell_{11}^d\}$ in
  Example~\ref{ex:instance}, $B\left(\ell^d_3,X\right)=\left\{\ell_3^d,\,\ell_4^u,\,\ell_5^u\right\}$.
\end{example}
We call two blocks $(\ell,B)$ and $(\ell',B')$ disjoint if $B \cap B' = \emptyset$. Every
feasible solution can be represented as a collection of pairwise disjoint blocks. This
representation is not unique, and in particular, the partition into singleton sets always works.
\begin{example}
  For the solution $X=\left\{\ell_3^d,\,\ell_4^u,\,\ell_5^u,\,\ell_9^u,\,\ell_{11}^d\right\}$ in
  Example~\ref{ex:instance},
  $X=\left\{\ell_3^d,\,\ell_4^u,\,\ell_5^u\right\}\cup\left\{\ell_9^u,\,\ell_{11}^d\right\}$ and
  $X=\left\{\ell_3^d,\,\ell_4^u\right\}\cup\{\ell_5^u\}\cup\left\{\ell_9^u,\,\ell_{11}^d\right\}$
  are block decompositions. 
\end{example}
We will show that an optimal collection of disjoint blocks can be found in polynomial time. For a
path $\ell\in\mathcal P$, let $H_L(\ell)$ and $H_R(\ell)$ be the sets of paths to the left of $\ell$
and to the right of $\ell$, respectively:
\begin{align*}
   H_L(\ell) &=
  \begin{cases}
    \{\ell^u_{i'} \in \mc{P}: i' < i\} \cup \{\ell^d_{i'} \in \mc{P}: i'\Delta+\delta < i\Delta \}  &\text{if } \ell=\ell^u_i \\
    \{\ell^d_{i'} \in \mc{P}: i' < i \} \cup \{\ell^u_{i'} \in \mc{P}: i'\Delta+\delta <  i\Delta \} &\text{if } \ell=\ell^d_i.     
  \end{cases},\\
  H_R(\ell) &=
  \begin{cases}
    \{\ell^u_{i'} \in \mc{P}: i' > i\} \cup \{\ell^d_{i'} \in \mc{P}:  i'\Delta > i\Delta+\delta \}  &\text{if } \ell=\ell^u_i \\
    \{\ell^d_{i'} \in \mc{P}: i' > i \} \cup \{\ell^u_{i'} \in \mc{P}:  i'\Delta > i\Delta +\delta   \} &\text{if } \ell=\ell^d_i.
  \end{cases}
\end{align*}
In addition set $H_L(0)=H_R(0)=\mathcal P$.
\begin{example}
  In Figure~\ref{fig:assumptions},
  $H_L\left(\ell^d_6\right)=\left\{\ell^d_1,\dots,\ell^d_5\right\}\cup\{\ell^u_1\}$ and $H_R\left(\ell^d_6\right)=\left\{\ell^d_7,\dots,\ell^d_{12}\right\}\cup\{\ell^u_{12}\}$.
\end{example}
For the block $(\ell,B)$ and the job set $J$, let
$J_M(J,B)$ be the set of jobs from $J$ that are covered by $B$, and let $J_L(J,\ell,B)$ and
$J_R(J,\ell,B)$ be the sets of jobs from $J$ that can be covered by paths to the left and the right
of $\ell$, respectively. More formally,
\begin{align*}
  J_M(J,B)&=\{j \in J: \mc{R}_j \cap B \neq \emptyset\}, \\
  J_L(J,\ell,B)&=\{j \in J \setminus J_M(J,B): \mc{R}_j \cap H_L(\ell) \neq \emptyset\},\\
  J_R(J,\ell,B)&=\{j \in J \backslash J_M(J,B): \mc{R}_j \cap H_R(\ell) \neq \emptyset\}. 
\end{align*}
Then $J_L(J,\ell,B)\cap J_R(J,\ell,B)=\emptyset$. To see this, suppose there is a job
$j\in J_L(J,\ell,B)\cap J_R(J,\ell,B)$ (in particular, $j\not\in J_M(J,B)$). Then
$\mathcal R_j\cap H_L(\ell)\neq\emptyset$ and $\mathcal R_j\cap H_R(\ell)\neq\emptyset$, which
implies $\ell\in\mathcal R_j$. But then $j\in J_M(J,B)$, which contradicts the assumption. We are
particularly interested in blocks $(\ell,B)$ with the property that every job that cannot be
scheduled in block $B$, can be scheduled by cancelling paths that are either to the left or to the
right of $\ell$. This is captured by the following definition.
\begin{definition}\label{def:splitting_block}
  A block $(\ell,B)$ is \emph{splitting} (for a job set $J\subseteq \mc{J}$) if
  \[J_M(J,B)\cup J_L(J,\ell,B)\cup J_R(J,\ell,B)=J.\]
\end{definition}
The only way for a block $(\ell_i^u,B)$ to violate the condition in
Definition~\ref{def:splitting_block} is that there exists a job $j\in J$ such that $\mathcal R_j$
does not contain any up-paths which implies $\mathcal R_j=\{\ell^d_{i'}\}$ for some $i'$ with
$\ell^d_{i'}\cap\ell^u_i\neq\emptyset$. In this situation $\ell^d_{i'}$ must be an element of $B$ if
$(\ell_i^u,B)$ is splitting. Similarly, a block $(\ell_i^d,B)$ is splitting for $J$ if and only
if $\ell^u_{i'}\in B$ for all $i'$ which satisfy $\ell^u_{i'}\cap\ell^d_{i}\neq\emptyset$ and 
$\mathcal R_j=\{\ell^u_{i'}\}$ for some $j\in J$.
\begin{example}
  In Figure~\ref{fig:assumptions} every block is splitting because every $\mathcal R_j$ contains
  paths in both directions.
\end{example}

\begin{lemma}\label{lem:splitting_block}
  Let $X \subseteq \mc{P}$ be a feasible solution for a job set $J \subseteq \mc{J}$, and let
  $\ell\in X$. Then $(\ell,B(\ell,X))$ is a splitting block for $J$.
\end{lemma}
\begin{proof}
  By feasibility of $X$, for every job in $j\in J\setminus J_M(J,B)$ there is a path $\ell'\in
  X\cap\mathcal R_j$. Now $j\not\in J_M(J,B)$ implies $\ell'\cap\ell=\emptyset$, hence $\ell'\in
  H_L(\ell)\cup H_R(\ell)$, and therefore $j\in J_L(J,\ell,B)\cup J_R(J,\ell,B)$. 
\end{proof}
\begin{definition}\label{def:leading_block}
  A block $(\ell,B)$ is called a \emph{leading} block for job set $J$ if it is splitting and
  $J_L(J,\ell,B)=\emptyset$.
\end{definition}
In other words, $\left(\ell^d_i, B\right)$ is a leading block if $B\cap\mathcal R_j\neq\emptyset$ for every $j$ with the property that
$\mathcal R_j$ does not contain any down-path $\ell^d_{i'}$ with $i'\geq i$.
.
\begin{example}
  In Figure~\ref{fig:assumptions}, $\left(\ell_4^d,B\right)$ is a leading block if $B$ contains
  up-paths intersecting the two jobs left of $\ell^d_4$. As a consequence, $\left(\ell_4^d,\left\{\ell_4^d,\,\ell_3^u,\,\ell_5^u\right\}\right)$ is not
  a leading block, but $\left(\ell_4^d,\left\{\ell_4^d,\,\ell_4^u\right\}\right)$ is.
\end{example}
Our algorithm is based on the following two observations.
\begin{enumerate}
\item Every feasible solution starts with a leading block.
\item If $(\ell,B)$ is a leading block for $\mathcal J$, then $B$ together with an optimal solution
  for the job set $J_R(\mathcal J,\ell,B)$ restricted to the path set $H_R(\ell)$ gives an optimal
  solution $X$ under the additional constraint that $B(\ell,X)=B$.
\end{enumerate}
The first of these observations will be formalized and proved in Lemma~\ref{lem:leading_block},
while the second one is the main idea behind the proof of Lemma~\ref{lem:block_recursion}.

Let $X$ be a feasible solution for $J \in \mc{J}$, and let $i_1=\min\{i\,:\,\ell^u_i\in X\}$,
$i_2=\min\{i\,:\,\ell^d_i\in X\}$, where we use the convention $\min\emptyset=\infty$. Let $\ell(X)$
be the \emph{first path} of $X$, which is defined as 
\[\ell(X)=
  \begin{cases}
    \ell^u_{i_1} &\text{if }i_1\leq i_2,\\
    \ell^d_{i_2} &\text{if }i_1>i_2.
  \end{cases}
\]
\begin{lemma}\label{lem:leading_block}
  Let $J \subseteq \mc{J}$ be a job set, let $X \subseteq \mc{P}$ be a feasible solution for $J$, let
  $\ell=\ell(X)$ be its first path, and let $B=B(\ell,X)$. Then $(\ell,B)$ is a leading block for
  $J$. Moreover, $X \setminus B \subseteq H_R(\ell)$.
\end{lemma}
\begin{proof}
  It follows from Lemma \ref{lem:splitting_block} that $(\ell,B)$ is a splitting block, and it
  remains to be shown that $J_L(J,\ell,B)=\emptyset$. Suppose there exists $j\in
  J_L(J,\ell,B)$. From the feasibility of $X$ and $\mathcal R_j\cap (B\cup H_R(\ell))=\emptyset$ we
  deduce that there exists a path $\ell'\in X\cap H_L(\ell)$, but this contradicts the assumption
  that $\ell=\ell(X)$ is the first path of $X$. The second part follows from
  \[X=(X\cap H_L(\ell))\cup(X\cap\{\ell'\in\mathcal P\,:\,\ell'\cap\ell\neq\emptyset\})\cup(X\cap
    H_R(\ell))=\emptyset\cup B\cup(X\cap
    H_R(\ell)). \qedhere\]
\end{proof}
\begin{example}
  For the solution $X=\{\ell_3^d,\,\ell_4^u,\,\ell_5^u,\,\ell_9^u,\,\ell_{11}^d\}$ in
  Example~\ref{ex:instance}, the first path is $\ell_3^d$ and the corresponding leading block is
  $\left(\ell_3^d,\left\{\ell_3^d,\ell_4^u,\ell_5^u\right\}\right)$.
\end{example}
Lemma \ref{lem:leading_block} implies that for every job set $J \subseteq \mc{J}$ and every feasible
solution $X$ for $J$, there exists a leading block $(\ell,B)$ with $B \subseteq X$ such that every
job in $J$ can be scheduled by cancelling either a path in $B$ or a path in $H_R(\ell)$. In the next
lemma, we prove the transitivity of the relation ``is to the right of'' on the set $\mathcal
P$, which will be used in the proof of Lemma~\ref{lem:split_lead_blocks}.
\begin{lemma}\label{lem:H_R_subset_H_R}
  For every $\ell\in\mathcal P$ and every $\ell'\in H_R(\ell)$, $H_R(\ell') \subseteq H_R(\ell)$.
\end{lemma}
\begin{proof}  
  There are four cases regarding the directions of the paths $\ell$ and $\ell'$.
  \begin{description}
  \item[Case 1] Both are up-paths, say $\ell=\ell^u_i$ and $\ell'=\ell^u_{i'}$. Then $\ell'\in
    H_R(\ell)$ implies $i'>i$, and
    \begin{multline*}
      H_R\left(\ell'\right)=\{\ell^u_{i''} \in \mc{P}: i'' > i'\} \cup \{\ell^d_{i''} \in \mc{P}:
      i''\Delta > i'\Delta+\delta \}\\
      \subseteq\{\ell^u_{i''} \in \mc{P}: i'' > i\} \cup \{\ell^d_{i''} \in \mc{P}:  i''\Delta > i\Delta+\delta \}=H_R(\ell).
    \end{multline*}    
  \item[Case 2] $\ell$ is an up-path and $\ell'$ is a down-path, say $\ell=\ell^u_i$ and
    $\ell'=\ell^d_{i'}$. Then $\ell'\in H_R(\ell)$ implies $i'\Delta >i\Delta+\delta$, and
    \begin{multline*}
      H_R\left(\ell'\right)=\{\ell^d_{i''} \in \mc{P}: i'' > i' \} \cup \{\ell^u_{i''} \in \mc{P}:  i''\Delta > i'\Delta +\delta\}\\
      \subseteq\{\ell^u_{i''} \in \mc{P}: i'' > i\} \cup \{\ell^d_{i''} \in \mc{P}:  i''\Delta > i\Delta+\delta \}=H_R(\ell).
    \end{multline*}    
  \item[Case 3] Both are down-paths, say $\ell=\ell^d_i$ and $\ell'=\ell^d_{i'}$. Then $\ell'\in
    H_R(\ell)$ implies $i'>i$, and
    \begin{multline*}
      H_R\left(\ell'\right)=\{\ell^d_{i''} \in \mc{P}: i'' > i'\} \cup \{\ell^u_{i''} \in \mc{P}:
      i''\Delta > i'\Delta+\delta \}\\
      \subseteq\{\ell^d_{i''} \in \mc{P}: i'' > i\} \cup \{\ell^u_{i''} \in \mc{P}:  i''\Delta > i\Delta+\delta \}=H_R(\ell).
    \end{multline*}    
  \item[Case 4] $\ell$ is a down-path and $\ell'$ is an up-path, say $\ell=\ell^d_i$ and
    $\ell'=\ell^u_{i'}$. Then $\ell'\in H_R(\ell)$ implies $i'\Delta >i\Delta+\delta$, and
    \begin{multline*}
      H_R\left(\ell'\right)=\{\ell^u_{i''} \in \mc{P}: i'' > i' \} \cup \{\ell^d_{i''} \in \mc{P}:  i''\Delta > i'\Delta +\delta\}\\
      \subseteq\{\ell^d_{i''} \in \mc{P}: i'' > i\} \cup \{\ell^u_{i''} \in \mc{P}:  i''\Delta > i\Delta+\delta \}=H_R(\ell).\qedhere
    \end{multline*}    
  \end{description}
\end{proof}

\begin{lemma}\label{lem:split_lead_blocks}
  Let $(\ell,B)$ be a block with $J:=J_R(\mc{J},\ell,B)\neq\emptyset$, and let $(\ell',B')$ be a
  leading block for~$J$. Then $ J_R(J,\ell',B')=J_R(\mc{J},\ell',B')$. 
\end{lemma}
\begin{proof}
  Lemma~\ref{lem:H_R_subset_H_R} implies $J_R(\mc{J},\ell',B')\subseteq J_R(\mc{J},\ell,B)=J$, and
  then
  \begin{multline*}
    J_R(\mc{J},\ell',B')=J_R(\mc{J},\ell',B')\cap J= \{j\in\mathcal J\setminus J_M(\mathcal J,B)\,:\,\mathcal R_j\cap
    H_R(\ell)\neq\emptyset\}\cap J\\
    = \{j\in J\setminus J_M(J,B)\,:\,\mathcal R_j\cap  H_R(\ell)\neq\emptyset\} = J_R(J,\ell',B').\qedhere
  \end{multline*}
\end{proof}
For $J \subseteq \mc{J}$, let $ \mc{P}^*(J) = \bigcup_{j \in J}\mc{R}^*_j $, and for a block
$(\ell,B)$, let $\mathcal P^*_R(\ell,B)=\mc{P}^*(J_R(\mc{J},\ell,B)) \cap H_R(\ell)$. The following lemma can
be proved by the same shifting argument as Lemma~\ref{lem:interesting_paths}.

\begin{lemma}\label{lem:interesting_paths_extension}
  For $\ell \in \mc{P}$, let $J \subseteq\{j\in\mathcal J\,:\,\mc{R}_j\cap H_R(\ell) \neq
  \emptyset\}$, and let
  \[\mu=\min\{\lvert X\rvert\,:\,X\subseteq H_R(\ell),\,X\cap\mathcal
    R_j\neq\emptyset\text{ for all }j\in J\}.\]
  There exists $X\subseteq \mathcal P^*_R(\ell,B)$ with $\lvert X\rvert=\mu$ and $X\cap\mathcal
    R_j\neq\emptyset$ for all $j\in J$. \hfill$\qed$
\end{lemma} 
For a block $(\ell,B)$ let
\[\varphi(\ell,B)=\min\{\lvert X\rvert\,:\,X\subseteq H_R(\ell),\,X\cap\mathcal
    R_j\neq\emptyset\text{ for all }j\in J_R(\mathcal J,\ell,B)\}.\]
In particular, $\varphi(0,\emptyset)$ is the optimal objective value for the instance $(\mathcal
J,\mathcal P)$ of the maintenance scheduling problem.
\begin{lemma}\label{lem:block_recursion}
  The function $\varphi$ satisfies the recursion
  \begin{equation}\label{eq:block_recursion:recursion}
    \varphi(\ell,B)=
    \begin{cases}
      \min\{|B'|+\varphi(\ell',B'): (\ell',B') \in F(\ell,B)\}   &\text{if } J_R(\mc{J},\ell,B)\neq \emptyset \\
      0 &\text{if } J_R(\mc{J},\ell,B)=\emptyset
    \end{cases} 
  \end{equation}
  where $F(\ell,B)=\{(\ell',B')\,:\,B' \subseteq \mc{P}^*_R(\ell,B),\, (\ell',B') \text{ is a leading block for }J_R(\mc{J},\ell,B)\}$.
\end{lemma}
\begin{proof}
  Let $J=J_R(\mc{J},\ell,B)$. For $J=\emptyset$, \eqref{eq:block_recursion:recursion} is
  trivial. For $J\neq\emptyset$, Lemma~\ref{lem:interesting_paths_extension} yields the existence of
  a set $X\subseteq \mathcal P^*_R(\ell,B)$ with $\lvert X\rvert=\varphi(\ell,B)$ and
  $X\cap\mathcal R_j\neq\emptyset$ for all $j\in J_R(\mc{J},\ell,B)$. Let $\ell^*=\ell(X)$ be the
  first path of $X$, and let $B^*=B(\ell^*,X)$ be the set of paths in $X$ that intersect
  $\ell^*$. By Lemma~\ref{lem:leading_block} and $X\subseteq \mathcal P^*_R(\ell,B)$, we have
  $(\ell^*,B^*)\in F(\ell,B)$. Moreover, $\lvert X\setminus B'\rvert\geq\varphi(\ell^*, B^*)$,
  because $X\setminus B^*\subseteq H_R(\ell^*)$ and $(X\setminus B^*)\cap\mathcal R_j\neq\emptyset$
  for all $j\in J_R(J,\ell^*,B^*)=J_R(\mathcal J,\ell^*,B^*)$, where the last equality follows from
  Lemma~\ref{lem:split_lead_blocks}. As a consequence,
  \[\varphi(\ell,B)= \lvert X\rvert=\lvert B^*\rvert+\lvert X\setminus B^*\rvert \geq \lvert
    B^*\rvert+\varphi(\ell^*,B^*) \geq \min\{\lvert B'\rvert +\varphi(\ell',B')\,:\,(\ell',B') \in
    F(\ell,B)\}.\]
  For the converse inequality, let $(\ell^*,B^*)\in F(\ell,B)$ be a minimizer for the right hand side
  of~\eqref{eq:block_recursion:recursion}. In particular $(\ell^*,B^*)$ is a leading block for $J$,
  hence, using Lemma~\ref{lem:split_lead_blocks} again,
  \[J=J_M(J,B^*)\cup J_R(J,\ell^*,B^*)=J_M(J,B^*)\cup J_R(\mathcal J,\ell^*,B^*).\]
  By the definition of $\varphi(\ell^*,B^*)$, there exists a set $X'\subseteq H_R(\ell^*)$ with $\lvert
  X'\rvert=\varphi(\ell^*,B^*)$ and $X\cap\mathcal R_j\neq\emptyset$ for all $j\in J_R(\mathcal
  J,\ell^*,B^*)$. Then $X=B^*\cup X'\subseteq H_R(\ell)$ satisfies $X\cap\mathcal R_j\neq\emptyset$ for
  all $j\in J$, hence
  \[\varphi(\ell,B)\leq\lvert X\rvert=\lvert B^*\rvert+\lvert X'\rvert=\lvert B^*\rvert+\varphi(\ell^*,B^*)=\min\{\lvert B'\rvert +\varphi(\ell',B')\,:\,(\ell',B') \in
    F(\ell,B)\}.\qedhere\]
\end{proof}
The minimum number of cancelled paths for the problem can be obtained by finding
$\varphi(0,\emptyset)$ using~\eqref{eq:block_recursion:recursion}. As a consequence, we can
formulate the problem as a shortest problem on the directed acyclic graph $G=(V,A)$ with node set
$V=V_0\cup\{O,D\}$, where $V_0=\{(\ell,B)\,:\,B \subseteq \mc{P}^*, (\ell,B) \text{ is a block} \}$,
and arc set $A = A_O\cup A_M \cup A_D$ where
\begin{align*}
  A_O &= \{(O,(\ell,B))\,:\,(\ell,B) \in F(0,\emptyset)\}, \\
  A_M &= \{((\ell,B),(\ell',B'))\,:\,(\ell,B) \in V_0,\,(\ell',B')\in F(\ell,B)\}, \\  
  A_D &= \{((\ell,B),D)\,:\, (\ell,B) \in V_0,\, J_R(\mc{J},\ell,B)= \emptyset\}.
\end{align*}
To each arc $(*,(\ell',B')) \in A_O \cup A_M$, we assign cost $\lvert B'\rvert$, while all arcs in
$A_D$ have cost zero. The following is an immediate consequence of Lemma~\ref{lem:block_recursion}.
\begin{corollary}\label{cor:shortest_path_formulation}
  The optimal objective value for the maintenance scheduling problem equals the length of a shortest
  path from node $O$ to node $D$ in $G=(V,A)$.\hfill$\qed$
\end{corollary}
The node set can be reduced further by the following two observations:
\begin{enumerate}
\item Without loss of generality, we can assume that if we have an optimal solution $X$ with leading
  block $\left(\ell,B\right)$ where $\ell$ is the first path of $X$. As a consequence, we only need
  to consider leading blocks $\left(\ell_i^d,B\right)$ with $i'>i$ for all $\ell^u_{i'}\in B$ and
  leading blocks $\left(\ell_i^u,B\right)$ with $i'\geq i$ for all $\ell^d_{i'}\in B$.
\item From any two blocks $(\ell,B)$ and $(\ell,B')$ with $J_M(J,B)=J_M(J,B')$ we need to include
  only one, and we can choose the smaller among the sets $B$ and $B'$ if their sizes differ.
\end{enumerate}
\begin{example}\label{ex:first_steps}
  For the instance in Figure~\ref{fig:assumptions}, there are four possible first paths to consider:
  $\ell_1^d$, $\ell_3^d$, $\ell_3^u$ and $\ell_4^u$. In Table~\ref{tab:first_steps} we list the corresponding
  ten out-neighbours of node $O$, together with the corresponding set of remaining jobs
  $J_R(\mathcal J,\ell,B)$.
  \begin{table}[htb]
    \centering
        \caption{The out-neighbours of node $O$ and the corresponding sets
      $\mathcal J_R(\mathcal J,\ell,B)$.}
    \label{tab:first_steps}
    \begin{tabular}{ll}\toprule
      leading block $(\ell,B)$ & $J_R(\mathcal J,\ell,B)$ \\\midrule
  $\left(\ell_1^d,\left\{\ell_1^d\right\}\right)$ & $\{2,3,4,5,6,7,8\}$\\
  $\left(\ell_1^d,\left\{\ell_1^d,\,\ell^u_5\right\}\right)$ & $\{3,6,7,8\}$\\
      $\left(\ell_3^d,\left\{\ell_3^d,\,\ell^u_4\right\}\right)$ & $\{4,6,7,8\}$\\
  $\left(\ell_3^d,\left\{\ell_3^d,\,\ell^u_4,\,\ell^u_5\right\}\right)$ & $\{6,7,8\}$\\
  $\left(\ell_3^u,\left\{\ell_3^u\right\}\right)$ & $\{1,2,4,6,7,8\}$\\
  $\left(\ell_3^u,\left\{\ell_3^u,\,\ell_3^d\right\}\right)$ & $\{1,4,6,7,8\}$\\
  $\left(\ell_3^u,\left\{\ell_3^u,\,\ell_5^d\right\}\right)$ & $\{1,2,6,7,8\}$\\ 
  $\left(\ell_3^u,\left\{\ell_3^u,\,\ell_3^d,\,\ell_5^d\right\}\right)$ & $\{1,6,7,8\}$\\  
  $\left(\ell_4^u,\left\{\ell_4^u,\,\ell_5^d\right\}\right)$ & $\{6,7,8\}$\\
  $\left(\ell_4^u,\left\{\ell_4^u,\,\ell_5^d,\,\ell_6^d\right\}\right)$ & $\{7,8\}$\\ \bottomrule
    \end{tabular}
  \end{table}
  Figure~\ref{fig:sp_network} shows three paths in the network including a shortest path which
  corresponds to an optimal solution for the instance.
  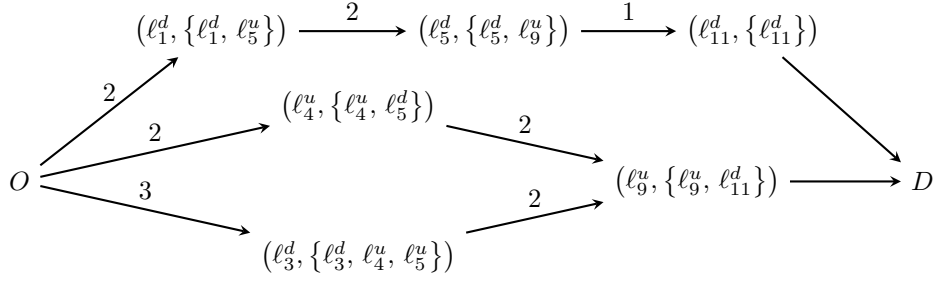
\begin{figure}[htb]
    \centering
    \begin{tikzpicture}[xscale=1.5]
      \node (O) at (0,0) {{\small $O$}};
      \node (D) at (8,0) {{\small $D$}};
      \node (v1) at (3,-1) {{\small
          $\left(\ell_3^d,\left\{\ell_3^d,\,\ell^u_4,\,\ell^u_5\right\}\right)$}};
      \node (v2) at (3,1) {{\small
          $\left(\ell_4^u,\left\{\ell_4^u,\,\ell^d_5\right\}\right)$}};
      \node (v3) at (1.7,2) {{\small
          $\left(\ell_1^d,\left\{\ell_1^d,\,\ell^u_5\right\}\right)$}};
      \node (w1) at (6,0) {{\small
          $\left(\ell_9^u,\left\{\ell_9^u,\,\ell^d_{11}\right\}\right)$}};
      \node (w2) at (4.2,2) {{\small
          $\left(\ell_5^d,\left\{\ell_5^d,\,\ell^u_9\right\}\right)$}};
      \node (x) at (6.5,2) {{\small
          $\left(\ell_{11}^d,\left\{\ell_{11}^d\right\}\right)$}};
      \draw[thick,->] (O) to node[above] {{\small $3$}} (v1);
      \draw[thick,->] (O) to node[above] {{\small $2$}} (v2);
      \draw[thick,->] (O) to node[above] {{\small $2$}} (v3);
      \draw[thick,->] (v2) to node[above] {{\small $2$}} (w1);
      \draw[thick,->] (v3) to node[above] {{\small $2$}} (w2);
      \draw[thick,->] (w2) to node[above] {{\small $1$}} (x);
      \draw[thick,->] (v1) to node[above] {{\small $2$}} (w1);
      \draw[thick,->] (w1) to (D);
      \draw[thick,->] (x) to (D);
    \end{tikzpicture}
    \caption{Three paths in the digraph for the instance in
      Figure~\ref{fig:assumptions}. These paths correspond to the solutions
      $X_1=\left\{\ell_1^d,\,\ell^u_5,\,\ell_5^d,\,\ell^u_{9},\ell^d_{11}\right\}$,
      $X_2=\left\{\ell_4^u,\,\ell^d_5,\,\ell_9^u,\,\ell^d_{11}\right\}$ and
      $X_3=\{\ell_3^d,\,\ell_4^u,\,\ell_5^u,\,\ell_9^u,\,\ell_{11}^d\}$, and $X_2$ is an optimal
      solution.}\label{fig:sp_network}
  \end{figure}

\end{example}

In the next lemma, we show that a block $(\ell',B')\in F(\ell,B)$ cannot be too far away from the
first path of $\mathcal P^*_R(\ell,B)$. For this purpose, we define a set $C(\ell)\subseteq\mathcal
P$, for every $\ell\in\mathcal P$ by 
\[C(\ell)=\begin{cases}
      \{\ell^d_i,\ell^d_{i+1},\dots,\ell^d_{i_1}\}\cup\{\ell^u_i,\ell^u_{i+1},\dots,\ell^u_{i_2}\}
      &\text{if }\ell=\ell^u_i,\\
      \{\ell^u_i,\ell^u_{i+1},\dots,\ell^u_{i_1}\}\cup\{\ell^d_i,\ell^d_{i+1},\dots,\ell^d_{i_2}\}
      &\text{if }\ell=\ell^d_i
    \end{cases}\]
  where $i_1=i_1(i)=\min\{k\,:\,k\Delta>i\Delta+\delta\}$ and $i_2=i_2(i)=\max\{k\,:\,k\Delta\leq
  i_1\Delta+\delta\}$, and note that $\lvert C(\ell)\rvert\leq 3(\delta/\Delta+1)=O(1)$.
\begin{lemma}\label{lem:leading_blocks}
  Let $(\ell,B)\in\{(0,\emptyset)\}\cup V_0$ with $J=J_R(\mc{J},\ell,B)\neq\emptyset$, and let $\ell^*=\ell(\mathcal
  P^*_R(\ell,B))$ be the first path of $\mathcal P^*_R(\ell,B)$. If $(\ell',B')$ is a leading block
  for $J$ then $\ell'\in C(\ell^*)$. In particular, $\lvert F(\ell,B)\rvert=O(1)$.
\end{lemma}
\begin{proof}
  We treat the case $\ell^*=\ell^u_i$, and note that the case that $\ell^*$ is a down-path can be
  treated similarly. Let $j\in J$ be a job for which $\ell^*=\ell^u_i$ is the last up-path in
  $\mathcal R_j$. Suppose the statement of the lemma is false. Since $\ell'\in\mathcal P^*_R(\ell,B)$ and
  $\ell^u_i$ is the first path of $\mathcal P^*_R(\ell,B)$, this implies that $\ell'=\ell^d_k$ with
  $k\geq i_1+1$ or $\ell'=\ell^u_k$ with $k\geq i_2+1$. Then
  \begin{align*}
    B' &\subseteq \left\{\ell^d_{i'}\,:\,i'\geq i_1+1\right\}\cup\left\{\ell^u_{i'}\,:\,i'\geq
         i+1\right\},&
    H_R(\ell) &\subseteq \left\{\ell^d_{i'}\,:\,i'\geq i_1+2\right\}\cup\left\{\ell^u_{i'}\,:\,i'\geq
         i_2+1\right\}.
  \end{align*}
  Since $\ell^u_i$ is the last up-path in $\mathcal R_j$, it follows that
  $\mathcal R_j\cap(B'\cup H_R(\ell))=\emptyset$, so
  $j\in J\setminus(J_M(J,B')\cup J_R(J,\ell',B'))$, which contradicts the assumption that
  $(\ell',B')$ is a leading block for $J$. The boundedness of $\lvert F(\ell,B)\rvert$ now follows from the observation that for every
  $\ell'\in C(\ell^*)$ there are at most $2^{1+2\delta/\Delta}=O(1)$ possible $B'$ to form an
  element $(\ell',B')\in F(\ell,B)$.
\end{proof}

As a corollary of Lemma~\ref{lem:leading_blocks}, we obtain that $G$ is sparse. 
\begin{corollary}\label{cor:bounded_out-degree}
  $\lvert V\rvert=O(n)$, and every node has out-degree $O(1)$. In particular, $\lvert A\rvert=O(n)$.
\end{corollary}
\begin{proof}
  For every $(\ell,B)\in V_0$, $\ell\in\mathcal P^*$, and $B\subseteq\{\ell'\in\mathcal
  P^*\,:\,\ell'\cap\ell\neq\emptyset\}$, which implies
  \[\lvert V\rvert=2+\lvert V_0\rvert\leq 2+\lvert \mathcal
    P^*\rvert2^{1+2\delta/\Delta}\leq2+2n2^{1+2\delta/\Delta}=O(n).\]  
  It follows from Lemma~\ref{lem:leading_blocks} that the out-degree of every node is $O(1)$. 
\end{proof}
\begin{theorem}\label{thm:bunded_intersections}
  The maintenance scheduling problem with bounded number of intersections and singleton possessions can be
  solved in time $ O(n^2) $.
\end{theorem}
\begin{proof}
  In Corollaries~\ref{cor:shortest_path_formulation} and~\ref{cor:bounded_out-degree}, we have
  reduced the problem to finding a shortest path in a directed acyclic digraph with $O(n)$ arcs,
  which can be done in time $O(n)$ (see, e.g.,~\cite[Theorem 4.3]{ahuja1993network}). In order to
  conclude the proof, we need to show that the digraph can be constructed in time $O(n^2)$. We
  initialise a queue containing only node $O$, and while the queue is not empty, we pick a node from
  the queue, construct its out-neighbours (as in Example~\ref{ex:first_steps} for node $O$), and
  remove the node from the queue. This is described more precisely in Algorithm~\ref{alg:sp_graph}.
  \begin{algorithm}
    \caption{Construction of the digraph $G=(V,A)$}\label{alg:sp_graph}
    \begin{tabbing}
      ....\=....\=....\=....\=............................ \kill \\[-2ex]
      \textbf{Input:} $\mathcal R^*_j$ for every $j\in\mathcal J$\\[1ex]
      Initialize $V\leftarrow\{O,D\}$, $A\leftarrow\emptyset$, $Q\leftarrow\{(0,\emptyset)\}$\\
      \textbf{while} $Q\neq\emptyset$ \textbf{do}\\
      \> pick $(\ell,B)\in Q$\\
      \> $J\leftarrow J_R(\mathcal J,\ell,B)$\\      
      \> $\mathcal P^*\leftarrow\mathcal P^*_R(\ell,B)$\\
      \> $\ell^*\leftarrow\ell(\mathcal P^*(\ell,B))$\\
      \> \textbf{for} $\ell'\in C(\ell^*)$ and $B'\subseteq\{\ell''\in\mathcal P^*\,:\,\ell''\cap\ell'\neq\emptyset\}$
      \textbf{do}\\
      \> \> \textbf{if} $(\ell',B')$ is a leading block for $J$ \textbf{then}\\
      \> \> \> $V\leftarrow V\cup\{(\ell',B')\}$\\
      \> \> \> $A\leftarrow A\cup\{((\ell,B),\,(\ell',B'))\}$\\
      \> \> \> \textbf{if} $J_R(\mathcal J,\ell',B')\neq\emptyset$ \textbf{then} $Q\leftarrow Q\cup\{(\ell',B')\}$\\
      \> \> \> \textbf{else} $A\leftarrow A\cup\{((\ell',B'),\,D)\}$\\
      \> $Q\leftarrow Q\setminus\{(\ell,B)\}$\\[1ex]
      \textbf{Output:} digraph $G=(V,A)$
    \end{tabbing}
  \end{algorithm}
  The while loop is processed once for each of the $O(n)$ nodes $(\ell,B)\in V$, and the result
  follows since each iteration of the loop takes time $O(n)$.
\end{proof}
\begin{example}
  For the instance in Figure~\ref{fig:assumptions} the construction of the the out-neighbours of node
  $\left(\ell_3^d,\left\{\ell_3^d,\ell_4^u\right\}\right)$ is illustrated in Figure~\ref{fig:out-neighbours}.
  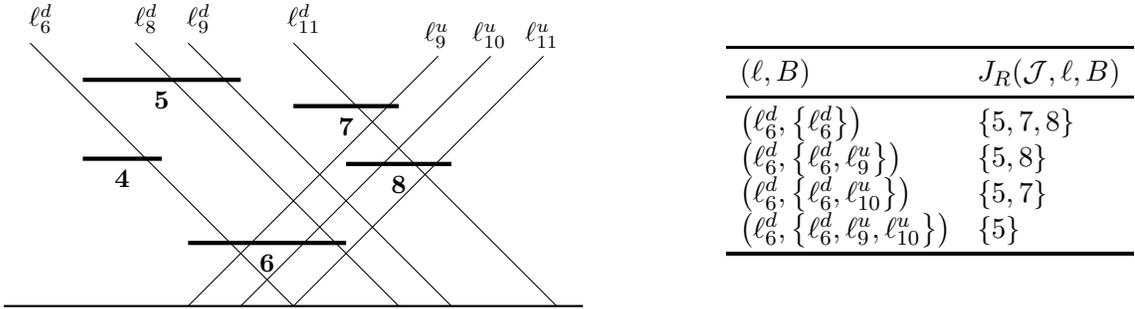
\begin{figure}[htb]
    \begin{minipage}{.6\textwidth}
      \centering
  \begin{tikzpicture}[scale=.7]
    \draw[thick,-] (3.5,0) -- (14.5,0);    
    \foreach \x in {6,8,9,11} {
      \draw (\x-2,5) -- (\x+3,0);
      \draw (\x-1.8,5) node[anchor=south] {{\small $\ell_{\x}^d$}};      
    }
   \foreach \x in {9,10,11} {
      \draw (\x-2,0) -- (\x+2.75,4.75);      
      \draw (\x+2.72,4.72) node[anchor=south] {{\small $\ell_{\x}^u$}};
    }
    \draw[ultra thick] (5,2.8) to node[below] {{\small$\mathbf{4}$}} (6.5,2.8);
    \draw[ultra thick] (5,4.3) to node[below] {{\small$\mathbf{5}$}} (8,4.3);
    \draw[ultra thick] (7,1.2) to node[below] {{\small$\mathbf{6}$}} (10,1.2);
    \draw[ultra thick] (10,2.7) to node[below] {{\small$\mathbf{8}$}} (12,2.7);
    \draw[ultra thick] (9,3.8) to node[below] {{\small$\mathbf{7}$}} (11,3.8);    
  \end{tikzpicture}
\end{minipage}\hfill
\begin{minipage}{.39\textwidth}
\centering
  \begin{tabular}{ll} \toprule
      $\left(\ell,B\right)$ & $J_R(\mathcal J,\ell,B)$ \\ \midrule
      $\left(\ell_6^d,\left\{\ell_6^d\right\}\right)$ & $\{5,7,8\}$\\
      $\left(\ell_6^d,\left\{\ell_6^d,\ell^u_9\right\}\right)$ & $\{5,8\}$\\
      $\left(\ell_6^d,\left\{\ell_6^d,\ell^u_{10}\right\}\right)$ & $\{5,7\}$\\
      $\left(\ell_6^d,\left\{\ell_6^d,\ell_9^u,\ell^u_{10}\right\}\right)$ & $\{5\}$\\ \bottomrule
    \end{tabular}
  \end{minipage}
\caption{The remaining problem for node $\left(\ell_3^d,\left\{\ell_3^d,\ell_4^u\right\}\right)$ and
the resulting out-neighbours.}\label{fig:out-neighbours}  
\end{figure}
\end{example}
\begin{remark}
  The arguments presented in this section show that the maintenance scheduling problem with
  singleton possessions remains polynomially solvable as long as the number of intersections is
  bounded by the logarithm of $n$, that is, Assumption~\ref{ass:bounded_intersections}, which is
  equivalent to $\delta/\Delta=O(1)$, can be relaxed to $\delta/\Delta=O(\log n)$. 
\end{remark}

\section{Computational experiments}\label{sec:computation}  
In this section, we compare MIP formulations on instances that include only up-paths. The
computational experiments were carried out on a Dell Latitude E5570 laptop with Intel Core i7-6820HQ
2.70GHz processor, and 8GB of RAM running Ubuntu 16.04. We used Python 3.5 and Gurobi 7.0.2 to solve
the MIP models. We set a time limit of 100 seconds for solving each instance. If an optimal solution
was not reached within this time frame, the best incumbent was reported.

We compare the strength of each formulation on 12 classes of randomly generated instances. Each
class contains 20 instances. An instance is given by a set of jobs $\mc{J}$, a set of paths
$\mc{P}$, and a set of collections~$\mc{R}_j$, $j \in \mc{J}$. Each class is represented by a pair
$(n,L)$ where $n$ is the number of jobs and $L$ is the maximum number of distinct lengths. The main
steps of the procedure for generating each instance of class $(n,L)$ are as follows: First, $L$
distinct lengths are chosen randomly from $\{1,2,3,5,7,11\}$. Second, the number of paths (i.e.,
$m$) is set equal to $n$ times the average of the chosen distinct lengths divided by 2. Third, the
length of each job is chosen randomly from the $L$ chosen distinct lengths. Finally, the first and
last path for each job are randomly chosen from $\{1,\dotsc, m\}$.

We compare the models based on the following measures: the number of instances that are solved to
optimality by the LP relaxation (N. LP Optimal), the average solution runtime (Avg. IP
runtime), the maximum solution runtime (Max IP runtime), the average of
\[\frac{|\text{Optimal LP relaxation obj. value - Optimal obj. value}|}{\text{Optimal obj. value}}\]
(Avg. LP Gap), the maximum LP Gap (Max LP Gap), the average of
\[\frac{|\text{Incumbent obj. value - Optimal obj. value}|}{\text{Optimal obj. value}}\]
(Avg. MIP Gap) where incumbent refers to a best integer solution found within 100 seconds, and
maximum MIP Gap (Max MIP Gap).

The results are presented in Tables~\ref{table:SCMModelStatistics} and~\ref{table:HSMModelStatistics}.
\begin{table}[htbp]
  \renewcommand{\arraystretch}{1.1}
  \centering
  \caption{Computational results for SCM formulation. All times are wall clock times and are given in seconds.}
  \begin{tabular}{@{}p{1.5cm}p{1.5cm}p{1.5cm}p{1.5cm}p{1.5cm}p{1.5cm}@{}}
    \toprule
    $(n,Nl)$ & N. LP Optimal & Avg. IP Runtime & Max IP Runtime & Avg. LP Gap (\%) & Max LP Gap (\%) \\
    \midrule
    (30,2) & 17    & 0.01  & 0.02  & 0.30 & 2.38 \\
    (30,3) & 18    & 0.01  & 0.02  & 0.34 & 5.26 \\
    (30,4) & 19    & 0.01  & 0.02  & 0.08 & 1.52 \\
    (50,2) & 18    & 0.06  & 0.13  & 0.07 & 0.76 \\
    (50,3) & 19    & 0.06  & 0.09  & 0.10 & 1.96 \\
    (50,4) & 16    & 0.07  & 0.11  & 0.26 & 1.67 \\
    (70,2) & 19    & 0.30  & 0.61  & 0.07 & 1.47 \\
    (70,3) & 17    & 0.34  & 1.11  & 0.12 & 1.19 \\
    (70,4) & 15    & 0.20  & 0.47  & 0.28 & 1.97 \\
    (100,2) & 18    & 1.90  & 4.30  & 0.13 & 1.68 \\
    (100,3) & 16    & 1.69  & 5.44  & 0.16 & 1.21 \\
    (100,4) & 17    & 1.24  & 2.86  & 0.10 & 0.79 \\
    \bottomrule
    \end{tabular}%
      \label{table:SCMModelStatistics}
    \end{table}%
    \begin{table}[htbp]
      \renewcommand{\arraystretch}{1.1}  \centering
      \caption{Computational results for PIM formulation. All times are wall clock times and are
        given in seconds.}
    \begin{tabular}{@{}p{1.5cm}p{1.5cm}p{1.5cm}lp{1.5cm}p{1.5cm}p{1.5cm}p{1.5cm}@{}}
    \toprule
    $(n,Nl)$ & Avg. IP Runtime & Max IP Runtime & Avg. LP Gap (\%) & Max LP Gap (\%) & Avg. MIP Gap (\%) & Max MIP Gap (\%) \\
    \midrule
   (30,2) & 46.55 & 100.21 & 51.18 & 80.23 & 0 & 0 \\
   (30,3) & 73.96 & 100.33 & 58.77 & 70.78 & 0.26 & 3.03 \\
   (30,4) & 76.44 & 100.19 & 58.06 & 76.57 & 0.11 & 2.17 \\
   (50,2) & 71.90 & 100.10 & 53.06 & 78.85 & 0.77 & 5.41 \\
   (50,3) & 81.63 & 100.59 & 55.33 & 75.23 & 0.58 & 5.00 \\
   (50,4) & 91.40 & 100.31 & 57.29 & 74.74 & 1.26 & 6.38 \\
   (70,2) & 90.16 & 100.25 & 59.41 & 77.52 & 3.95 & 22.12 \\
   (70,3) & 96.40 & 100.10 & 60.28 & 74.72 & 3.55 & 9.72 \\
   (70,4) & 100.01 & 100.09 & 56.17 & 73.81 & 2.99 & 9.84 \\
   (100,2) & 91.12 & 101.73 & 58.48 & 76.56 & 11.70 & 50.00 \\
   (100,3) & 100.01 & 100.09 & 60.73 & 77.91 & 10.66 & 32.47 \\
   (100,4) & 100.03 & 100.23 & 61.37 & 73.42 & 5.52 & 15.79 \\
    \bottomrule
    \end{tabular}%
  \label{table:HSMModelStatistics}
\end{table}%
The LP relaxation of the PIM formulation did not yield the optimal value for any instance whereas
the LP relaxation of SCM formulation gave the optimal value for the majority of instances. The
maximum LP Gap for the SCM formulation on all instances is $5.26\%$ whereas that of the PIM
formulation is $80.23\%$. The maximum average LP Gap for the SCM formulation on all classes is
$0.34\%$ compared to $61.37\%$ for the PIM formulation. The MIP gap on all instances for SCM
formulation is zero whereas the maximum average and maximum MIP gap for PIM formulation on all
classes are $11.70\%$ and $50.00\%$. It is clear from the computational results that SCM formulation
outperforms PIM formulation considerably. We did some additional experiments for larger instances,
and found that the SCM formulation typically finds the optimal solution within the time limit of
$100$ seconds for up to 380 jobs.

\section{Conclusions and open problems}\label{sec:conclusion}

We have introduced the problem of scheduling maintenance jobs in a railway
corridor with bidirectional traffic so that the number of train paths
that have to be cancelled due to the maintenance are minimized. We
proved that the general problem is NP-complete, presented two integer
programming formulations, and showed that the problem can be solved in
polynomial time using dynamic programming in two special cases:
\begin{enumerate}
\item The train paths are only in a single direction.
\item The number of paths crossing any given path is bounded, and the
  maintenance jobs are short in the sense that each job can be
  scheduled to cancel only one path.
\end{enumerate}

The maintenance scheduling problem studied in this paper is a tactical
planning problem. While we have made some simplifying assumptions
which align with standard industry practice, it is natural to
investigate what happens when we remove some of these simplifying
assumptions. For example, in reality, train paths do not correspond to
equidistant parallel lines in the plane. Rather, train paths in
operational plans are represented by piecewise linear monotone
functions (increasing for up-paths and decreasing for down-paths). The
integer programming models in Section~\ref{sec:MIP} can easily capture
this setting as they do not depend on the geometry of the train
paths. On the other hand, the polynomial time algorithm for the
unidirectional case presented in Section~\ref{sec:unidirectional}
crucially depends on the assumption that the train paths are
equidistant. This motivates the following problem.
\begin{problem}
  Can the unidirectional maintenance scheduling problem be solved in
  polynomial time if paths correspond to collections of parallel
  straight lines, but we remove the assumption that they are
  equidistant?
\end{problem}
      
Another interesting direction is a polyhedral study of the [SCM]
formulation. In Section~\ref{subsec:IP} we observed that in the
unidirectional case the LP relaxation is integral if all of the jobs
have the same length.
\begin{problem}
  Find strong valid inequalities for the sets of feasible solutions to
  the problems [uniSCM] and [SCM]. More ambitiously, characterize the
  convex hulls of these sets, possibly under additional assumptions.
\end{problem}

In the practical maintenance scheduling setting it might be desirable
that there is a balance between the number of cancelled up-paths and
cancelled down-paths. This could also be addressed in future work.

\printbibliography

\end{document}